\documentclass[12pt]{amsart}
\usepackage{amsfonts}
\usepackage{dsfont}
\usepackage{bm}
\usepackage{comment}
\usepackage{bbm}
\usepackage{fourier}

\newtheorem{mainthm}{Theorem}

\newtheorem{thm}[equation]{Theorem}
\newtheorem{prop}[equation]{Proposition}
\newtheorem{lemma}[equation]{Lemma}
\newtheorem{cor}[equation]{Corollary}

\theoremstyle{definition}
\newtheorem{defi}[equation]{Definition}

\newtheorem{rem}[equation]{Remark}

\theoremstyle{remark}

\renewcommand{\theequation}{\thesection.\ifnum\value{subsection}=0 1\else \arabic{subsection}\fi.\arabic{equation}}

\usepackage{comment}

\usepackage[usenames]{color}
\usepackage[table,xcdraw]{xcolor}
\usepackage{tikz}
\usepackage[framemethod=tikz]{mdframed}
\definecolor{aliceblue}{rgb}{0.92, 0.93, 1.0}

\usepackage{mathtools}

\usepackage[colorlinks]{hyperref}

\newcounter{change}

\usepackage[colorinlistoftodos, textwidth=4cm]{todonotes}

\begin{document}

%%
%% The title of the paper goes here.  Edit to your title.
%%

%\title{Besov spaces of distributions as particle systems}
\title{Anisotropic spaces for the bilateral shift}\footnote{ This preprint  has not undergone peer review  or any post-submission improvements or corrections. The Version of Record of this article will be  published in {\it \href{https://link.springer.com/journal/10884}{Journal of Dynamics and Differential Equations}}, and will be  available online at \url{http://dx.doi.org/10.1007/s10884-025-10474-y} }

%%
%% Now edit the following to give your name and address:
%% 

 % Delete if not wanted.

%%
%% If there is another author uncomment and edit the following.
%%

\author{Mateus Marra}
\email{mateus-marra@usp.br}
\author{Daniel Smania}
\email{smania@icmc.usp.br}
\urladdr{https://sites.icmc.usp.br/smania/}

\address{Departamento de Matemática, Instituto de Ciências Matemáticas e de Computação (ICMC), Universidade de São Paulo (USP), Avenida Trabalhador São-carlense, 400,  
São Carlos, SP, CEP 13566-590, Brazil}

%%
%% If there are three of more authors they are added in the obvious
%% way. 
%%

%%%
%%% The following is for the abstract.  The abstract is optional and
%%% if not used just delete, or comment out, the following.
%%%

\begin{abstract} 
Given two H\"older potentials \( \phi_+ \) and \( \psi_- \) for the unilateral shift, we define anisotropic Banach spaces of distributions on the bilateral shift space with a finite alphabet. On these spaces, the transfer operator for the bilateral shift is quasicompact with a spectral gap, and the unique Gibbs state associated with \( \phi_+ \) spans its \( 1 \)-eigenspace. This result allows us to establish exponential decay of correlations 
for H\"older observables and a wide range of measures on the bilateral shift space.
\end{abstract}
%%
%%  LaTeX will not make the title for the paper unless told to do so.
%%  This is done by uncommenting the following.
%%

 \maketitle

%%
%% LaTeX can automatically make a table of contents.  This is done by
%% uncommenting the following:
%%

\setcounter{tocdepth}{1}
\tableofcontents

%%
%%  To enter text is easy.  Just type it.  A blank line starts a new
%%  paragraph. 
%%

%%
%% A new section is started as follows:
%%

%Seu resumo aqui, com no máximo 10 páginas

\section{Introduction}

Dynamical systems with hyperbolic behavior often possess a large number of invariant measures. For smooth dynamics, however, particular attention is given to {\it SRB measures} (see Eckmann and Ruelle \cite{er}), which are invariant measures that describe the statistical behavior of most orbits in the Lebesgue sense (or with respect to some volume form on a manifold). This concept is closely related to {\it physical measures}, although the literature varies in definition, often because these two measures coincide.

Much of the initial success in smooth ergodic theory involved studying (sufficiently regular) one-dimensional (piecewise) expanding maps. In these cases, SRB measures are absolutely continuous with respect to Lebesgue measure. For Markovian maps, the problem can be reduced to studying the {\it Gibbs states} of the unilateral shift, where Ruelle’s pioneering work on thermodynamic formalism \cite{ruelle} applies. {\it Transfer operators} acting on function spaces play a crucial role in this area (see Parry and Pollicott \cite{parry1990zeta}).

A similar approach was used for Anosov diffeomorphisms by Ruelle and Bowen \cite{bowen}, employing Markov partitions developed by Sinai. Here again, the problem is reduced to methods involving the transfer operator applied to the unilateral shift. Transfer operator methods have also been highly effective in studying non-Markovian piecewise expanding maps in one or more dimensions. See Baladi \cite{baladi2} and the brief historical overview in \cite{arbieto2019transfer} for further details.

The work of Blank, Keller, and Liverani \cite{bkl} marked the first application of transfer operator methods directly to the transfer operator of an invertible map ($C^\infty$ Anosov diffeomorphism). One of the challenges in this setting is that the SRB measure can be {\it singular} with respect to the manifold’s volume form. Another challenge arises from differing behaviors in distinct directions (expansion/contraction), making conventional norms and function spaces unsuitable. These challenges were addressed by considering the transfer operator on {\it anisotropic spaces of distributions}. An extensive body of results, employing a variety of anisotropic spaces, followed. Baladi’s survey and book \cite{quest}\cite{baladibook} provide a comprehensive account. See also Demers, Kiamari, and Liverani \cite{impa} and Demers \cite{demers}.

Notably, new anisotropic spaces were introduced, which were suitable for studying (piecewise) smooth hyperbolic maps with lower regularity. Of particular note are studies on two-dimensional piecewise hyperbolic maps (Demers and Liverani \cite{dl1}), finite-horizon Sinai billiards (Demers \cite{dbook}), the measure of maximal entropy for piecewise hyperbolic maps \cite{d3}, the time-one map of a Lorentz gas (Demers, Melbourne, and Nicol \cite{d4}),  the measure of maximal entropy for finite-horizon Sinai billiard maps (Baladi and Demers \cite{d5}), and  hyperbolic set with a  isolating neighborhood (Baladi and Tsujii \cite{bt}).

An open question is whether anisotropic Banach spaces can be used in non-smooth settings. This is not immediately clear from previous work, as these approaches often rely on the smooth structure of the manifold and dynamics, such as smooth invariant foliations, Fourier analysis  techniques, microlocal analysis, smooth norms, and restrictions on the dimension/regularity of invariant manifolds.

{\it It turns out they can}. In this work, we demonstrate that it is possible to define anisotropic Banach spaces for a symbolic bilateral space with a finite number of symbols. This approach enables the study of the transfer operator acting on this space, allowing us to establish significant statistical properties of Gibbs states for the bilateral shift.

\section*{Acknowledgements} 

M.M. was supported by CAPES-Brazil. D.S.  was financed  by the S\~ao Paulo Research Foundation (FAPESP), Brasil, Process Number 2017/06463-3, and Bolsa de Produtividade em Pesquisa CNPq-Brazil 311916/2023-6.

\section{Bilateral shift and Gibbs states}
\subsection{Shifts} Let $\mathcal{A}=\{1,\dots,n \}$. For $j\in \mathbb{Z}$ define
$\pi_j\colon \mathcal{A}^\mathbb{Z}\rightarrow \mathcal{A},$
 as $\pi_j((x_k)_{k\in \mathbb{Z}})=x_j.$
Consider the bilateral shift
$\sigma\colon \mathcal{A}^\mathbb{Z}\rightarrow \mathcal{A}^\mathbb{Z}$
defined by $\pi_j(\sigma(x))=\pi_{j+1}(x)$ for all $j\in \mathbb{Z}$.
Denote $I=\mathcal{A}^\mathbb{Z}$.
Consider the unilateral shifts
$\sigma_+\colon \mathcal{A}^\mathbb{N} \rightarrow \mathcal{A}^\mathbb{N}$
$$\sigma_+(x_0,x_1,x_2,x_3,\dots)=(x_1,x_2,x_3,\dots),$$
and $\sigma_-\colon \mathcal{A}^{-\mathbb{N}^\star} \rightarrow \mathcal{A}^{-\mathbb{N}^\star}$
$$\sigma_-(\dots,x_{-3},x_{-2},x_{-1})=(\dots,x_{-3},x_{-2})$$
In particular we can define the inverse branches
\begin{align*}&\sigma_{+,x_{-k},x_{-(k-1)},\dots,x_{-1}}^{-1}(x_0,x_1,x_2,\dots)
=(x_{-k},x_{-(k-1)},\dots,x_{-1},x_0,x_1,x_2,\dots ),\\
&\sigma_{-,x_0.x_1,\dots,x_k}^{-1}(\dots,x_{-3},x_{-2},x_{-1})=(\dots,x_{-2},x_{-1},x_0,x_1,\dots,x_k).
\end{align*}

\subsection{Bilateral shift as a skew-product}
Of course we can identify
$$\mathcal{A}^\mathbb{Z}=\mathcal{A}^{-\mathbb{N}^\star} \times \mathcal{A}^\mathbb{N}.$$
Let $I_+=\mathcal{A}^\mathbb{N}$ and $I_-=\mathcal{A}^{-\mathbb{N}^\star}$. Define the projections
$$\pi_+((x_k)_{k\in \mathbb{Z}})=(x_k)_{k\in \mathbb{N}}, $$
$$\pi_-((x_k)_{k\in \mathbb{Z}})=(x_k)_{k\in -\mathbb{N}^\star}, $$
and see the bilateral shift as a skew-product
$\sigma\colon I_+\times I_-\rightarrow I_+\times I_-$
$$\sigma(x,y)=(\sigma_+(x),\sigma_{-,\pi_0(x)}^{-1}(y)).$$
Moreover $$\sigma^{-1}(x,y)=(\sigma_{+,\pi_{-1}(y)}^{-1}(x),\sigma_-(y)).$$
\subsection{Cylinders and words} Consider the {\it cylinders}
\begin{align*} C(x_{-m},\dots,x_{-1},x_0,x_1,\dots,x_k)&=\{y \in I\colon y_i=x_i \ for \ -m\leq i\leq k     \},\\
C_-(x_{-m},\dots,x_{-1})&=\{y \in I_-\colon y_i=x_i \ for \ -m\leq i\leq 1     \},\\
C_+(x_0,x_1,\dots,x_k)&=\{y \in I_+\colon y_i=x_i \ for \ 0\leq i\leq k     \}.
\end{align*}

The family of all cylinders in $I^+$ ($I^-$) will be denoted by $\mathcal{C}_+$ ($\mathcal{C}_-$, resp.).

A {\it word} is a finite sequence of symbols
$$\omega=x_0 x_1\dots x_\ell$$
with $x_i \in \mathcal{A}$. We denote the {\it length} $\ell+1$  of $\omega$ by $|\omega|$.

\subsection{Topology and metric} We endow $I, I^+$ and $I^-$ with the product topology. The following metric on $I^+$ 
$$d(x,y)=n^{-\min\{i\in \mathbb{N}\colon x_i\neq y_i   \}}.$$
induces the product topology on $I^+$ and it turns $\sigma_+$ into an expanding map. We can define  similar metrics on $I^-$ and $I^+\times I^-$ and we will use $d$ to represent all of them. 
\subsection{Gibbs state for the bilateral shift} Given a $\beta$-H\"older function $\phi\colon I \rightarrow \mathbb{R}$  there is a unique $\sigma$-invariant probability $\nu$ and  $C>0$ such that for every cylinder
$$C_+(x_0,x_1,\dots,x_{k-1})\subset I_+$$
and every $y \in C_+(x_0,x_1,\dots,x_{k-1})\times I_-$ we have
$$\frac{1}{C} \leq \frac{\nu(C_+(x_0,x_1,\dots,x_{k-1})\times I_-)}{\exp(-kP(\phi)+\sum_{i< k} \phi(\sigma^iy))}  \leq C$$
The  probability $\nu$ is called the {\it Gibbs state} associated to the potential $\phi$. Here $P(\phi)$ denotes the topological pressure of $\phi$. Without  loss of generality we assume $P(\phi)=0$.

Our goal is to study the statistical properties of $\nu$ using transfer operator methods associated to the bilateral shift.

\subsection{Gibbs states for unilateral shifts} 
For every H\"older function $\psi_\theta$, with $\theta~\in~\{+,-\}$, defined on $I^\theta$, one can consider the {\it transfer operator}
$$\mathcal{L}_{\psi_\theta,\sigma_\theta}\colon C(I^\theta) \rightarrow C(I^\theta) $$
defined by 
$$\mathcal{L}_{\psi_\theta,\sigma_\theta}h(x) = \sum_{\sigma_\theta(y)=x} e^{\psi_\theta(y)}h(y).$$
Its dual operator has an unique probability $m^{\phi_\theta}_\theta$ satisfying 
$$\mathcal{L}_{\psi_\theta,\sigma_\theta}^\star m^{\psi_\theta}_\theta=e^{P(\psi_\theta)}m^{\psi_\theta}_\theta$$
and there is an unique positive H\"older function $\rho_{\psi_\theta}$ such that 
\begin{align*} \mathcal{L}_{\psi_\theta,\sigma_\theta} \rho_{\psi_\theta}&=e^{P(\psi_\theta)}\rho_{\psi_\theta},\ and \\
\int \rho_{\psi_\theta} \ dm^{\psi_\theta}_s&= 1.
\end{align*} 

Then $\nu_\theta^{\psi_\theta}=\rho_{\psi_\theta} m^{\psi_\theta}_\theta$ is the   Gibbs state for the potential $\psi_\theta$ and the unilateral shift $\sigma_\theta$, that is, a $\sigma_\theta$-invariant probability such that there is $C> 0$ satisfying 

 $$\frac{1}{C} \leq \frac{\nu_\theta^{\psi_\theta}(C_\theta(\omega))}{\exp(-kP(\psi_\theta)+\sum_{i< k} \psi_\theta(\sigma^i_\theta y))}  \leq C$$
for  every cylinder $C_\theta(\omega)\subset I^\theta$ of length $k$ and every $y \in C_\theta(\omega)$. Moreover we have that $\rho_{\psi_\theta}$ is H\"older continuous and there is $C> 0$ such that 
\begin{equation}\label{similar} \frac{1}{C}\leq \rho_{\psi_\theta}(x)\leq C\end{equation} 
for every $x \in I^\theta$. 

\begin{rem}  The Gibbs state  $\mu_\theta^{\psi_\theta}$ is {\it unique and  mixing}.  Furthermore $\nu_\theta^{\psi_\theta}$ is the unique $\sigma_\theta$-invariant probability that is absolutely continuous with respect to $m_\theta$. See Parry and Pollicott \cite{parry1990zeta} for a full account.\end{rem} 

After proper normalization and  without loss of generality we assume that $P(\psi_\theta)=0$  for all potentials in this work. Note that the  reference measure  $m_\theta^{\psi_\theta}$ is such that 
for every cylinder $C_\theta(\omega)\subset I^\theta$ with minimal length one we have

$$m_\theta(\sigma_\theta(C_\theta(\omega)))=\int 1_{\sigma_\theta(C_\theta(\omega))} \ dm_\theta= \int e^{-\psi_\theta} 1_{C_\theta(\omega)}   \ dm_\theta,$$
that is,  $e^{-\psi_\theta}$ is the {\it jacobian} of the measure $m_\theta$ with respect to the injective branches of  $\sigma_\theta$.

\subsection{Constructing Gibbs states}   Firstly we can find  H\"older potentials $\phi_+$ and $\phi_-$ that are {\it cohomologous}  to $\phi$ and only depends on the non-negative (negative, respectively)  entries of the sequences, that is, there is a H\"older function $u_+$, $u_-$ such that 
\begin{equation}\label{coho} \phi_+ = \phi + u_+\circ \sigma - u_+\end{equation} 
$$\phi_- = \phi + u_-\circ \sigma - u_-$$
and \begin{itemize}
    \item[-] $\phi_+(x)=\phi_+(y)$ whenever $\pi_+(x)=\pi_+(y)$.
    \item[-] $\phi_-(x)=\phi_-(y)$ whenever $\pi_-(x)=\pi_-(y)$.
\end{itemize}
See Parry and Pollicott \cite{parry1990zeta} for a  proof of the existence of $\phi_+$ and $\phi_-$. Since $P(\phi)=0$ we have $P(\phi_+)=P(\phi_-)=0$. We can easily see that the Gibbs states of $\phi$ and $\phi_+$  are the same. \textit{Note  that every Gibbs state \(\nu\) of \(\sigma\) with respect to \(\phi\) must necessarily be a \(\sigma\)-invariant measure with marginal \(\pi_+^*(\nu) = \nu_+^{\phi_+}\).}

For details of the (very short) classical construction of the  Gibbs state using $\nu_+^{\phi_+}$ see Bowen \cite{bowen}). The   uniqueness of the Gibbs state $\nu$ follows from the fact that the invariant measure $\nu$ satisfies $$\nu(C_+(\omega)\times I^-)= \nu^{\phi_+}_+(C_+(\omega)),$$ that allows us to univocally deduce the measure of every other cylinder in $I^+\times I^-$. 

Se also  Climenhaga, Pesin,  Zelerowicz \cite{cpz}  for a detailed historical account on Gibbs measures and the many distinct ways to study them.

\subsection{Product structure of Gibbs states and SRB measures} The section will not be used to obtain the main results, so the reader can skip it if she/he  wish so. However we think it is useful to motivate our results. In particular we show that  Gibbs states have a product structure (see Haydn \cite{Haydn} and Leplaideur \cite{Leplaideur}). 

\begin{defi} Let  $\mu$ be a Borelian  probability  on $I^+\times I^-$. Let $\Omega_{\nu,\mu}^+$ be the set  of points $x \in I^+\times I^-$ such that  
$$\lim_k \frac{1}{k} \sum_{i< k} \delta_{\sigma^ix} =\nu$$
in the weak-$\star$ topology in  the dual of the continuous functions on $I^+\times I^-$.
We say that the Gibbs state $\nu$ is a   {\bf 
 forward SRB measure} with respect to  $\mu$ if $\mu(\Omega_{\nu,\mu}^+)> 0$.  In an analogous way 
let $\Omega_{\nu,\mu}^-$ be the set  of points $x \in I^+\times I^-$ such that  
$$\lim_k \frac{1}{k} \sum_{i< k} \delta_{\sigma^{-i}x} =\nu$$
 in the weak-$\star$ topology. We say that the Gibbs state $\nu$ is a   {\bf 
 backward  SRB measure} with respect to  $\mu$ if $\mu(\Omega_{\nu,\mu}^-)> 0$.
\end{defi}

We defer the proof of the next proposition to the Appendix.

\begin{prop} \label{srbp}
Let $\psi_+\colon I^+\rightarrow \mathbb{R}$ and $\psi_-\colon I^-\rightarrow \mathbb{R}$ be 
H\"older continuous functions with zero  topological pressure. Let $m=m_+^{\psi_+}\times m_-^{\psi_-}$. 
\begin{itemize}
\item[A.] If $\psi_+$ is not cohomologous to $\phi$ then $\nu$ is singular with respect to $m$ and moreover $\nu$ {\it is not } a forward SRB measure with respect to $m$.
\item[B.] If $\psi_+$ is cohomologous to $\phi$  but $\psi_-$ is not cohomologous to $\phi$ then $\nu$ is singular with respect to $m$, $\nu$  is  a forward SRB measure with respect to $m$ with $m(\Omega_{\nu,m}^+)=1$, but $\nu$  is not  a backward  SRB measure with respect to $m$.
\item[C.]  If $\psi_+$ and $\psi_-$ are  cohomologous to $\phi$  then  $\nu$ is absolutely  continuous with respect to  $m$ and $\nu$  is  a forward and backward  SRB measure with respect to $m$ with $m(\Omega_{\nu,m}^+)=m(\Omega_{\nu,m}^-)=1$.
\end{itemize}
Recall that $\phi$, $\phi_+$ and $\phi_-$ are cohomologous to each other, so $m_+^{\phi_+}\times m_-^{\phi_-}$ is in case $C$. 
\end{prop}

\subsection{Do you want to understand the past or the future? You need to choose} 
To apply the transfer operator approach in the study of the Gibbs measure \( \nu \), it is necessary to identify a \textbf{reference measure} \( m \) on \( I^+ \times I^- \) that is \textbf{almost invariant} under the push-forward action of the bilateral shift \( \sigma \).

The product measures \( m = m_+^{\psi_+} \times m_-^{\psi_-} \) in Proposition \ref{srbp} are quasi-invariant, and, as we will see, the associated transfer operator for these measures has a relatively simple structure. This structure facilitates nearly independent analysis along the unstable and stable "directions," \( I^+ \) and \( I^- \), making these measures well-suited for our study.

However, Proposition \ref{srbp}.A implies that if \( \psi_+ \) is not cohomologous to \( \phi \), then \( \nu \) provides little, if any, information about the \textbf{future} behavior of a typical point with respect to \( m = m_+^{\psi_+} \times m_-^{\psi_-} \). Therefore, {\it if we are concerned with future behavior, as we assume here, essentially our only choice is} to set \( \psi_+ = \phi_+ \).

The product structure of \( \nu \) described in Proposition \ref{srbp}.C makes the measure \( m = m_+^{\phi_+} \times m_-^{\phi_-} \) an appealing choice for the reference measure, as \( \nu \) then captures both the past and future behavior of \( m \)-typical points.

On the other hand, Proposition \ref{srbp}.B indicates that the Gibbs states \( \nu \) provide insights into the {\it future} dynamics of \( \sigma \) with respect to more general reference measures, though not necessarily about past dynamics. Hence, {\it if we are interested \textit{\textbf{only}} in the future, as we assume here}, we may adopt a more flexible approach for the stable direction \( I^- \) by choosing an \textbf{arbitrary} Hölder potential \( \psi_- \) with vanishing topological pressure and setting \( m = m_+^{\phi_+} \times m_-^{\psi_-} \). This flexibility allows us to derive statistical properties for these measures, such as the exponential decay of correlations.

\section{Main results: Spectral gap and Exponential decay of correlations}

Let $m_+=m_+^{\phi^+}$.

\begin{mainthm}[Spectral Gap for Anisotropic Banach Spaces]\label{main0}
Let $m_-=m_-^{\psi_-}$ be a reference measure associated with an arbitrary H\"older potential $\psi_-$ for $\sigma_-$, with zero topological pressure. For \  $0 < s < t< 1$, we use the reference measure $m = m_+ \times m_-$ to define:
\begin{itemize}
\item[i.] A metric $d_+$ on $I^+$ that is H\"older equivalent to the usual metric on $I^+$.
\item[ii.] A metric $d_-$ on $I^-$ that is H\"older equivalent to the usual metric on $I^-$.
\item[iii.] A space of functions $B^s_{1,1}(I^+)$, analogous to the usual Besov space with corresponding parameters, which is continuously embedded  in $L^r(m_+)$ for $r$ satisfying 
$$s - 1 + \frac{1}{r} = 0.$$
Additionally, we have $C^\beta(I^+,d_+) \subset B^s_{1,1}(I^+)$ for every $\beta > s$, where $C^\beta(I^+,d_+)$ denotes the space of H\"older continuous functions with respect to the metric $d_+$ but it  also includes certain discontinuous and unbounded functions.
\item[iv.] A space of distributions $B^{-t}_{1,1}(I^-)$, analogous to the usual distributional Besov space with corresponding parameters, which satisfies
$$(B^{-t}_{1,1}(I^-))^\star = C^t(I^-,d_-).$$
\item[v.] An anisotropic Banach space of distributions
$$B^{s,-t} = B^{s}_{1,1}(I^+) \otimes B^{-t}_{1,1}(I^-)$$
with the following properties:
\begin{itemize}
\item[-] If $f \in B^{s}_{1,1}(I^+)$ and $g\in B^{-t}_{1,1}(I^-)$ then $fg\in B^{s}_{1,1}(I^+) \otimes B^{-t}_{1,1}(I^-)$.
\item[-] The space generated by measures that are absolutely continuous with respect to $m$ and have densities given by characteristic functions of cylinders on $I^+ \times I^-$ is dense in $B^{s,-t}$.
\item[-] Every Borelian function  $\rho$ on $I^+ \times I^-$ such that there is $C$ such that  for every $x, x_0\in  I^+$ and $y,y_0 \in I^-$ we have
\begin{align*} &|\rho(x,y)-\rho(x_0,y)|\leq Cd_+(x,x_0)^{2t},\\
&|\rho(x,y)-\rho(x,y_0)|\leq Cd_-(y,y_0)^{2t}
\end{align*} 
is a bounded multiplier in $B^{s,-t}$.
\item[-] Let $r^\star$ be such that 
$$\frac{1}{r}+\frac{1}{r^\star}=1.$$
The distributions $\mu \in B^{s,-t}$ can be evaluated on every Borel measurable function $\gamma\colon I^+ \times I^- \to \mathbb{C}$ such that there is $K_{s,t}(\gamma)$ satisfying 
\begin{enumerate}
\item[-] We have 
$$||\gamma(\cdot,y)||_{L^{r^\star}(I^+)}\leq K_{s,t} (\gamma)$$
\item[-] We have
$$||\gamma(\cdot,y)- \gamma(\cdot,y_0)||_{L^{r^\star}(I^+)}\leq K_{s,t}(\gamma) d_-(y,y_0)^t.$$
for every $y,y_0\in I^-$ and $x\in I^+$.
\end{enumerate}
\end{itemize}
\item[vi.] The measure $m = m_+ \times m_-$ is $\sigma$-quasi-invariant, so its associated transfer operator
$$\mathcal{L}\colon L^1(m) \to L^1(m)$$
is well-defined.
\end{itemize}

There are $s^+$ and $t^-$ such that  if  \ $ 0 < s < t< 1$, $s < s^+$ and $t < t^-$  then:
\begin{itemize}
\item[A.] The intersection $L^1(m) \cap B^{s,-t}$ is $\mathcal{L}$-invariant and dense in both $L^1(m)$ and $B^{s,-t}$. Moreover, $\mathcal{L}$ extends uniquely to a bounded operator on $B^{s,-t}$.
\item[B.] For $\mu \in B^{s,-t}$, if $\mu$ is a Borel regular measure on $I^+ \times I^-$, then $\mathcal{L}\mu = \sigma^\star\mu$.
\item[C.] $\mathcal{L}$ is a bounded operator on $B^{s,-t}$ with a spectral radius of $1$.
\item[D.] $\mathcal{L}$ is a quasi-compact operator on $B^{s,-t}$, and its essential spectral radius satisfies
$$r_e(\mathcal{L}) \leq \max \{\exp(s M_+(\phi_+)), \exp(t M_-(\psi_-))\} < 1,$$
where
\begin{align*}
M_+(\phi_+) &= \sup_{\mu \text{ $\sigma_+$-invariant prob}} \int \phi_+ \, d\mu, \\
M_-(\psi_-) &= \sup_{\mu \text{ $\sigma_-$-invariant prob. }} \int \psi_- \, d\mu.
\end{align*}
\item[E.] $1$ is the only element of the spectrum of $\mathcal{L}$ on $B^{s,-t}$ lying on the unit circle. The generalized $1$-eigenspace of $\mathcal{L}$ is one-dimensional and is generated by the unique Gibbs state of $\phi_+$.
\item[F.] There is $C$ and $\lambda \in (0,1)$ such that for every $\mu \in B^{s,-t}$, functions $\rho$ and $\gamma$ as in (v) and for every $k\geq 0$
$$\bigg|\langle \mathcal{L}^k( \rho\mu),\gamma\rangle - \langle \mu, \rho \rangle   \langle  \nu,\gamma\rangle  \bigg|\leq C\lambda^k K_{s,t}(\gamma)||\mu ||_{B^{s,t}}$$
\end{itemize}
\end{mainthm}

We postpone the definition of $B^{s,-t}$ and technical details in the statement to Section \ref{ba}. Theorem \ref{main0} and additional results on the spaces $B^{s,-t}$ allows us to study statistical properties of the Gibbs state $\nu$. 

Theorem \ref{main0}.F is a decay of correlation result, however some measures are relatively {\it tame} elements on $B_{s,-t}$, so we can consider more general functions $\rho$ in this case.

Theorem \ref{main0} can be done in a similar way to a subshift of finite type with usual transitivity assumptions. We prefer to prove it to  the full shift so simplify the notation.

\begin{mainthm}[Exponential Decay of Correlations I] \label{main1} There are $s^+$ and $t^-$ such that  if  \ $ 0 < s < t< 1$, $s < s^+$ and $t < t^-$ then there is $C> 0$ and $\lambda < 1$ such that the following holds. Given an {\bf arbitrary} Borel probability  $\mu$ on $I^-$, and   functions $\rho, \gamma$ on $I^+\times I^-$ such that 
\begin{itemize}
\item[A.] The function $\rho$ is  Borel measurable  functions be such that there exists  \( C \) satisfying
$$ \int \Vert \rho(\cdot,y) \Vert_{B^{s}_{1,1}} \ d\mu(y) < \infty.$$ 
\item[B.] The function $\gamma\colon I^+\times I^-\rightarrow \mathbb{C}$ satisfies the assumptions in Theorem \ref{main0}.v
\end{itemize}
then $$\rho\  m_+\!\!\times \mu \in B^{s}_{1,1}(I^+)\otimes B^{-t}_{1,1}(I^-)$$ and  for every $k\geq 0$
$$\bigg|\int \gamma\circ \sigma^k \rho  \ d (m_+\!\!\times \mu) - \int \rho d (m_+\!\!\times \mu)\int \gamma \ d\nu     \bigg|\leq C\lambda^k K_{s,t}(\gamma)\int \Vert \rho(\cdot,y) \Vert_{B^{s}_{1,1}} d\mu(y)$$
\end{mainthm}
\ \\
Given $\beta > 0$ define the space $\dot{C}^\beta(I^+,I^-)$ of $\beta$-H\"older functions  $u\colon I^+\rightarrow I^-$  
with the pseudo-norm 
$$||u||_{\dot{C}^\beta(I^+,I^-)}= \sup_{\substack{ x,y\in I^+\\ x\neq y}} \frac{d_-(u(x),u(y))}{d_+(x,y)}.$$

 \begin{mainthm}[Exponential Decay of Correlations II]  \label{main7} 
There is $s^+$ and $t^-$ such that if   \( s \),  \( t \), $\beta$ and $\epsilon$ satisfies 
\begin{itemize}
    \item[i.] \( 0 < s < s_+ \), \( 0 < t < t_- \), and \( 0 < s < t <1\),
    \item[ii.] \( 0 < s < \beta t \) , $\epsilon\in (0,1)$ and  \( 0 < \epsilon < \beta t - s \).
\end{itemize} 
then we can chosse $C=C(s,t,\beta,\epsilon)$ and $ \lambda=\lambda(s,t) \in (0,1)$ such that the following holds. 
 Given 
 \begin{itemize} 
 \item[A.] An H\"older function $u \in \dot{C}^\beta(I ^+,I^-)$, 
\item[B.] An function \( \rho \in B^s_{1,1}(I^+)\cap L^{1/\epsilon}(I^+) \),
 \item[C.] A function $\gamma\colon I^+\times I^-\rightarrow \mathbb{C}$ satisfying the assumptions in Theorem \ref{main0}.v
\end{itemize}
Let $U\colon I^+\rightarrow I^+\times I^-$ be  the function  $U(x)=(x,u(x))$. Then  
\[
U^\star(\rho m_+) \in B^{s}_{1,1}(I^+) \otimes B^{-t}_{1,1}(I^-).
\]
and for every $k \geq 0$ 
\begin{align*} &\bigg|\int \gamma\circ \sigma^k  \ dU^\star(\rho m_+) - \int \rho dm_+ \int \gamma \ d\nu     \bigg|\\
&\leq C\lambda^k (\Vert\rho\Vert_{B^s_{1,1}(I^+)} + ||u||_{\dot{C}^\beta(I^+,I^-)}\Vert\rho\Vert_{L^{1/\epsilon}(m_+)}) K_{s,t}(\gamma).
\end{align*}
\end{mainthm}

In particular if $\rho$ is non negative then  $U^\star(\rho m_+)$ is a Borelian finite measure  supported on the graph of a H\"older function.

\subsection{Do we really need distributions?}

One might ask whether a space of distributions, rather than a space of functions, is truly necessary in Theorem \ref{main0}. Indeed, it is essential. First, when \( \psi^- \) is not cohomologous to \( \phi^- \), the Gibbs state \( \nu \) is not absolutely continuous with respect to \( m_+^{\phi_+} \times m_-^{\psi_-} \) (See Proposition \ref{srbp}.B. In other words, \( \nu \not\in L^1(m_+^{\phi_+} \times m_-^{\psi_-}) \) by abuse of notation), making it impossible to obtain a spectral gap in a Banach space of functions \( B \) that is continuously embedded in \( L^1 \), as this would imply the contrary.

However, there is a more basic reason why a spectral gap cannot exist for a Banach space of functions continuously embedded in \( L^1 \), even in this scenario. The issue lies in the injectivity of \( \sigma \).

\begin{prop}
Consider an \textit{injective} measurable dynamical system \( T\colon X \rightarrow X \) with a quasi-invariant measure \( m \) on a measurable space $(X,\mathcal{A})$. Let \( \mathcal{L}_T \) denote the transfer operator associated with \( (T,m) \). Suppose there exists a Banach space of functions \( B \subset L^1(m) \) such that:
\begin{itemize}
    \item[i.] \( B \) is continuously embedded in \( L^1(m) \),
    \item[ii.] \( B \) is dense in \( L^1(m) \),
    \item[iii.] \( B \) is invariant under \( \mathcal{L}_T \), \( \mathcal{L}_T \) is bounded on \( B \), and its essential spectral radius is less than \( 1 \).
\end{itemize}
Then \( (X,\mathcal{A},m) \) is atomic with a finite number of atoms
\end{prop}

\begin{proof}
The injectivity of \( T \) implies that
\begin{equation} \label{fff} 
|\mathcal{L}_T^k f|_{L^1(m)} = |f|_{L^1(m)}
\end{equation}
for \textit{every} \( f \in L^1(m) \). From (\ref{fff}) and property (i), we conclude that the spectrum of \( \mathcal{L}_T \) on \( B \) lies on the unit circle and consists of finitely many isolated eigenvalues, each with a finite-dimensional generalized eigenspace. Consequently, \( B \) must be finite-dimensional, and from property (ii), it follows that \( L^1(m) \) is finite-dimensional. Thus $(X,\mathcal{A},m)$ is atomic with a finite number of atoms.
\end{proof}

\section{Measure space with a good grid}

Let $(I, \mathcal{A}, m)$ be a measure space, with $m(I) = 1.$ We will use the notation $|\cdot| = m(\cdot)$ when it is clear which measure is being used.

\begin{defi}
    Consider $0 < \lambda_1 < \lambda_2 < 1,$ and we call $\mathcal{P} = (\mathcal{P}^n)_{n \in \mathbb{N}}$ a good grid (see Arbieto and Smania \cite{arbieto2019transfer}) if it has the following properties:
\begin{enumerate}
    \item $\mathcal{P}^0 = \{I\}.$
    \item $I = \cup_{Q \in \mathcal{P}^n} Q$ (up to a set of measure zero).
    \item The elements of the family $\{Q\}_{Q \in \mathcal{P}^n}$ are pairwise disjoint (up to a set of measure zero).
    \item Given $Q \in \mathcal{P}^n$ and $n > 0$, there exists $P \in \mathcal{P}^{n-1}$ such that $Q \subseteq P.$
    \item Given $Q \subseteq P$ satisfying $Q \in \mathcal{P}^{n+1}$ and $P \in \mathcal{P}^n$ for some $n \geq 0$, we have
    $$\lambda_1 \leq \frac{|Q|}{|P|} \leq \lambda_2.$$
    \item The family $\cup_n \mathcal{P}^n$ generates the $\sigma$-algebra $\mathcal{A}.$
\end{enumerate}
\end{defi}
The constants $\lambda_1$ and $\lambda_2$ describe the geometry of the good partition $\mathcal{P} = (\mathcal{P}^n)_{n \in \mathbb{N}} = \cup_n \mathcal{P}^n$. By making simple adjustments to the good grid, we can eliminate the need for the "up to a set of measure zero" assumption, which we will consider as having been done.

There is a quite wide set of examples of measure spaces with a good grid.

\begin{rem}\label{gibbs}  The example that will be used here are the measure spaces $(I_\theta,m_\theta^{\psi_\theta},\mathcal{C}_\theta)$, with $\theta\in \{+,-\}$, where $m_\theta^{\psi_\theta}$ is a Gibbs reference measure to the potential $\psi_\theta$. When the chosen $\psi_\theta$ is  clear we will denote  $|A|_\theta=m_\theta^{\psi_\theta}(A)$.
\end{rem}

\section{Banach Spaces} \label{ba} 
\subsection{Unbalanced Haar wavelets} Given $P \in\mathcal{P}$, let $\mathcal{H}_P$ be a collection of pairs $(L,R)$ with the following properties
\begin{itemize}
\item[A.] If $(L,R)\in \mathcal{H}_P$ then $L$ and $R$ are non empty collections of children of $Q$.
\item[B.] If $(L,R)\in \mathcal{H}_P$ then $\cup L$ and $\cup R$ are disjoint.
\item[C.] If $(L,R)\in \mathcal{H}_P$ and  $L$ has more than one child then there is exactly a pair $(L',R') \in \mathcal{H}_P$ such that $L=(\cup L')\cup (\cup R')$.
\item[D.] If $(L,R)\in \mathcal{H}_P$ and  $R$ has more than one child then there is exactly a pair $(L',R') \in \mathcal{H}_P$ such that $R=(\cup L')\cup (\cup R')$.
\item[E.] If $(L',R')\in \mathcal{H}_P$ then either $P=(\cup L')\cup (\cup R')$ or there is $(L,R) \in  \mathcal{H}_Q$ such that either $R=L'\cup R'$ or $L=L'\cup R'$
\end{itemize} 
If $\mathcal{P}$ is not a binary tree (that is, there are elements with at least three children) then you may have many choices for $\mathcal{H}_P.$ Let $\mathcal{H}=\{I\}\cup \cup_{P\in \mathcal{P}} \mathcal{H}_P$. 

Given $Q\in \mathcal{H}\setminus \{I\}$, let $R_Q$ and $L_Q$ be  such that  $Q=(L_Q,R_Q)$. Let $S\subset \mathcal{P}$. We define 
$$|S|=\cup_{A\in S} |A|.$$

We define the {\bf $\bf{(s,p)}$-atom} $a_Q: I \rightarrow \mathbb{R}$ with support in $Q$ is
\begin{equation}\label{6.74}
a_Q = |Q|^{s+1-1/p} \left(\frac{1_{L_Q}}{|L_Q|} - \frac{1_{R_Q}}{|R_Q|}\right),
\end{equation}
Here $1_S$ denotes the indicator function of the set $\cup S$. Define $a_I=1_I$.

\begin{figure}[h]
\includegraphics[scale=0.5]{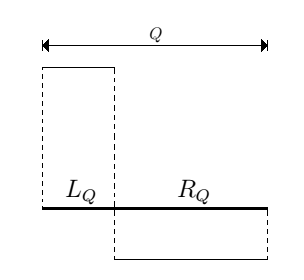}
\caption{An atom $a_Q$.}
\end{figure}

Note that $a_Q$ are pairwise orthogonal on $L^2(m)$. Indeed we have that 
$$\{a_Q\}_{Q\in \mathcal{H}}$$
is an unconditional Schauder basis of  $L^p(m)$, for every $p > 1$ (Girardi and Sweldens \cite{Sweldens1997}).

\subsection{The space of functions $B^s_{\infty,\infty}$} Let $s > 0$.  Define $B^s_{\infty,\infty}$ as the space of all function $\psi \in L^\infty(m)$ that can be written as
\begin{equation}\label{ad} \psi = \sum_{n \in \mathbb{N}} \sum_{P \in \mathcal{P}^n}\sum_{Q \in \mathcal{H}_P}  c_Q a_Q,\end{equation}
where $a_Q$ is a $(s,\infty)$ atom on $Q$, and 
$$\Vert\psi\Vert_{B^s_{\infty,\infty}}=\sup_{n\in \mathbb{N}} \sup_{P\in \mathcal{P}^n}  \sup_{Q\in \mathcal{H}_P} |c_Q|< \infty.$$

$B^s_{\infty,\infty}$ endowed with the norm $|\cdot|_{B^s_{\infty,\infty}}$ is a Banach space.

\begin{rem}\label{metric}  Define the pseudo-metric 
$$d_{I,\mu,\mathcal{P}}(x,y)= \inf_{\substack{ x,y \in P\\ P\in \mathcal{P}}} |P|.$$
 The space  $B^s_{\infty,\infty}$ is exactly the space of $s$-H\"older continuous functions $C^s(I,d_{I,\mu,\mathcal{P}})$ with respect to $d$ (see \cite{mms}) endowed with the usual norm. In fact, the good grid property implies that  the  space
$$\bigcup_{s> 0} B^s_{\infty,\infty}$$
{\it does not depend}  on the particularly chosen measure $\mu$ and consists exactly in the space of all H\"older functions (regardless of the exponent).  In fact if $(I,\mu,\mathcal{P})$ and $(I,\hat{\mu},\mathcal{P})$ are good grids then there is $\beta > 0$ and $C >  0$ such that 
$$d_{I,\mu,\mathcal{P}}(x,y)\leq C d_{I,\hat{\mu},\mathcal{P}}(x,y)^\beta  $$
for $x,y \in I$, that is, all those metrics are {\it H\"older equivalents}.  
\end{rem}

\subsection{The space of functions $B^s_{1,1}$} Let $s \in (0,1)$ Define $B^s_{1,1}$ as the space of all functions $\psi \in L^1(m)$ that can be written as in (\ref{ad}) 
where $a_Q$ is a $(s,1)$ atom on $Q$, and 
$$\Vert\psi\Vert_{B^s_{1,1}}=\sum_{n\in \mathbb{N}} \sum_{P\in \mathcal{P}^n}  \sum_{Q\in \mathcal{H}_P} |c_Q|< \infty.$$

Introduced in S. \cite{smania2022besov}, the set $B^s_{1,1}$, endowed with the norm $|\cdot|_{B^s_{1,1}}$ is also a Banach space of functions.

\begin{rem} It is easy to see that for $\beta > s$ $$B^\beta_{\infty,\infty}\subset B^s_{1,1} \subset L^1(m).$$
and that those inclusions are  continuous. 
\end{rem} 
The following will be useful 

\begin{prop}[\cite{smania2022besov}] \label{pos} Let $s\in (0,1)$. There is $C > 0$ such that for every $P \in \mathcal{P}$ we have
$$\Big| \frac{1_P}{|P|}  \Big|_{B^{s}_{1,1}}\leq C|P|^{-s}.$$
\end{prop}

\subsection{The space of distributions $B^{-s}_{1,1}$}

Define the space of {\bf test functions} $V_0$  as the space of all functions $\psi$ that can be written as (\ref{ad}) such that $a_Q$ is a $(0,1)$-atom on $Q$ and only a {\it finite} number  of coefficients $c_Q$ is non vanishing. A {\bf distribution} is a linear functional on $V_0$.

Let $s\in (0,1)$. The Besov space $B^{-s}_{1,1}$, introduced in Marra, Morelli and Smania \cite{mms},  is the space of distributions that can be written as in (\ref{ad}) 
where $a_Q$ is a $(-s,1)$ atom on $Q$, and 
$$\Vert\psi\Vert_{B^{-s}_{1,1}}=\sum_{n\in \mathbb{N}} \sum_{P\in \mathcal{P}^n}  \sum_{Q\in \mathcal{H}_P} |c_Q|< \infty.$$
$B^{-s}_{1,1}$ is a Banach space with such norm. We have

\begin{prop}[\cite{mms}]\label{dirac2} Let $s\in (0,1)$. There is $C > 0$ such that for every $P \in \mathcal{P}$ we have
$$\Big| \frac{1_P}{|P|}  \Big|_{B^{-s}_{1,1}}\leq C.$$
Indeed, given $x_0\in I$ and a sequence $P_n \in \mathcal{P}^n$ such that $x\in P_n$ for every $n$ we have that the limit 
$$\delta_{x_0}= \lim_n \frac{1_{P_n}}{|P_n|}$$
exists on $B^{-s}_{1,1}$. Furthermore if $x,y \in P\in \mathcal{P}$ we have
$$|\delta_x-\delta_y|_{B^{-s}_{1,1}}\leq C|P|^s.$$
\end{prop}
On can asks if we can extend a distribution $\psi$ as a bounded functional of a wider class of functions. More precisely given a measurable function $g\colon I \rightarrow \mathbb{C}$ we can ask when the limit 

$$\langle \psi,g\rangle = \lim_{N\rightarrow \infty} \sum_{n\leq N} \sum_{P \in \mathcal{P}^n}\sum_{Q \in \mathcal{H}_P}  \int  c_Q a_Q g \ dm$$
exists. 

\begin{prop}[\cite{mms}] Let $s\in (0,1)$. Given $g\in B^s_{\infty,\infty}(I)$ we have 
$$|\langle \psi,g\rangle|\leq \Vert\psi\Vert_{B^{-s}_{1,1}} \Vert g\Vert_{B^{s}_{\infty,\infty}}.$$
Indeed $(B^{-s}_{1,1}(I))^\star= B^s_{\infty,\infty}(I).$
\end{prop}

\subsection{The space of distributions $B^{-s}_{\infty,\infty}$}    
Let $s\in (0,1)$. We define  $B^{-s}_{\infty,\infty}$  as the space of distributions that can be written as in (\ref{ad}) 
where $a_Q$ is a $(-s,\infty)$ atom on $Q$, and 
$$\Vert\psi\Vert_{B^{-s}_{\infty,\infty}}=\sup_{n\in \mathbb{N}} \sup_{P\in \mathcal{P}^n}  \sup_{Q\in \mathcal{H}_P} |c_Q|< \infty.$$ 
$B^{-s}_{\infty,\infty}$  is a Banach space with such norm. We have

\begin{prop} The dual of $B^s_{1,1}$ is $B^{-s}_{\infty,\infty}$.
\end{prop}
\begin{proof} Note that $B^s_{1,1}$ is isometric to $\ell^1(\mathbb{N})$ (just map the unconditional Haar basis to the canonical basis of $\ell^1(\mathbb{N})$). Now it is easy to see  that its dual is $B^{-s}_{\infty,\infty}$.

\end{proof} 

\subsection{Product spaces and  Anisotropic Besov Spaces}
Let $(I^+,m_+,\mathcal{P}_+)$ and $(I^-,m_-,\mathcal{P}_-)$ be measurable spaces with good grids. One  can consider the  measure space $I^+\times I^-$ endowed the good grid  $$\mathcal{P}_+\times \mathcal{P}_- = (\mathcal{P}_+^n\times \mathcal{P}_-^n)_n.$$
In particular we can consider test functions and distributions on $$(I^+\times I^-,m_+\times m_-,\mathcal{P}_+\times \mathcal{P}_-).$$ Let $s,t \in (0,1)$. We introduce the {\bf anisotropic Besov space} 
$$B^{s}_{1,1}(I^+)\otimes B^{-t}_{1,1}(I^-)$$
in the following way. A distribution $\psi$ belongs to this space if and only if it can be written as

\begin{equation}\label{atomic} \psi = \sum_{n \in \mathbb{N}} \sum_{P \in \mathcal{P}^n_+}\sum_{Q \in \mathcal{H}_P} \sum_{m \in \mathbb{N}} \sum_{R \in \mathcal{P}^m_-}\sum_{J \in \mathcal{H}_R}     c_{Q,J} a_Q b_J = \sum_{Q\in \mathcal{H}^+} \sum_{J\in \mathcal{H}^-} c_{Q,J} a_Q^+ b_J^-, \end{equation}
where $a_Q^+$ is the $(s,1)$-atom supported on $Q$ an $b_J^-$ is the $(-t,1)$-atom supported on $J$, and such that 
\begin{align*} &\Vert\psi\Vert_{B^{s}_{1,1}(I^+)\otimes B^{-t}_{1,1}(I^-)}\\
&= \sum_{n \in \mathbb{N}} \sum_{P \in \mathcal{P}^n_+}\sum_{Q \in \mathcal{H}_P} \sum_{m \in \mathbb{N}} \sum_{R \in \mathcal{P}^m_-}\sum_{J \in \mathcal{H}_R}     |c_{Q,J}| =\sum_{Q\in \mathcal{H}^+} \sum_{J\in \mathcal{H}^-} |c_{Q,J}|< \infty\end{align*}

We have that $$B^{s}_{1,1}(I^+)\otimes B^{-t}_{1,1}(I^-)$$ is a Banach space of distributions and 
$$\{a_Q^+b_J^-\}_{\substack{Q\in \mathcal{H}^+\\J\in \mathcal{H}^-}}$$
is an unconditional Schauder basis of it. Note that we can define pseudo-metrics $d_+$ on $I^+$ and $d_-$ on $I^-$ as in Remark \ref{metric}.

One can easily see that 
\begin{prop}\label{prodd}  If $a^+ \in B^{s}_{1,1}(I^+)$ and $b^-\in B^{-t}_{1,1}(I^-)$ then $$a^+b^- \in B^{s}_{1,1}(I^+)\otimes B^{-t}_{1,1}(I^-),$$ and
$$\Vert a^+b^-\Vert_{B^{s}_{1,1}(I^+)\otimes B^{-t}_{1,1}(I^-)}=\Vert a^+\Vert_{B^{s}_{1,1}(I^+)}\Vert b^-\Vert_{B^{-t}_{1,1}(I^-)}. 
$$
\end{prop} 

The space of distributions $$B^{s}_{1,1}(I^+)\otimes B^{-t}_{1,1}(I^-)$$ is quite wide, however it is easy to see that 
\begin{prop} The push-forward by the projection $\pi_+$   of a distribution $\mu \in  B^{s}_{1,1}(I^+)\otimes B^{-t}_{1,1}(I^-)$  is a finite complex-valued Borelian measure that is absolutely continuous with respect to $m_+$ and 
$$\Big|\frac{d\pi_+^\star(\mu)}{dm_+}\Big|_{B^{s}_{1,1}(I^+)}\leq \Vert\mu\Vert_{B^{s}_{1,1}(I^+)\otimes B^{-t}_{1,1}(I^-)}.$$
\end{prop}

\begin{prop}\label{holdermult}  Suppose $0< s < t< 1$. Let $\psi \in B^{s}_{1,1}(I^+)\otimes B^{-t}_{1,1}(I^-)$ and $\rho\colon I^+\times I^-\rightarrow \mathbb{C}$ such that 
\begin{itemize}
\item[A.] $\rho \in L^\infty(m_+\times m_-)$,
\item[B.] There is $C$ and $\beta >0$ such that  for every $x, x_0\in  I^+$ and $y,y_0 \in I^-$ we have
\begin{align*} &|\rho(x,y)-\rho(x_0,y)|\leq Cd_+(x,x_0)^{2t},\\
&|\rho(x,y)-\rho(x,y_0)|\leq Cd_-(y,y_0)^{2t}.
\end{align*}
\end{itemize} 
Then $\rho \psi \in B^{s}_{1,1}(I^+)\otimes B^{-t}_{1,1}(I^-)$.
\end{prop} 

We postpone the proof of Proposition \ref{holdermult} to Section \ref{ultima}.

\subsection{Evaluation of distributions}  \label{evaluation} Distributions in $ B^{s}_{1,1}(I^+)\otimes B^{-t}_{1,1}(I^-)$ can be evaluated on functions in $I^+\times I^-$ that are regular enough. Let $r\in (1,\infty]$ such that 
$$s-1+\frac{1}{r}=0$$
and let $r^\star \in (1,\infty]$ be defined by 
$$\frac{1}{r}+\frac{1}{r^\star}=1.$$

Given a function
$$\gamma\colon I^+ \times I^-\rightarrow \mathbb{C}$$
such that there is $K(\gamma)$ satisfying 
$$||\gamma(\cdot,y)||_{L^{r^\star}(I^+)}\leq K(\gamma)$$
and
$$||\gamma(\cdot,y)- \gamma(\cdot,y_0)||_{L^{r^\star}(I^+)}\leq K(\gamma) d_-(y,y_0).$$
for every $y,y_0\in I^-$ and $x\in I^+$. Then since $B^s_{1,1}(I^+)\subset L^{r^\star}(I^+)$ we have that 
$$y\mapsto \int \gamma(x,y)a^+(x)\ dm_+(x)$$
 is $t$-H\"older continuous with respect to the metric $d_-$ on its domain, and  since $B^{-t}_{1,1}(I^-)^\star=C^t(I^-,d_-)$  we can  estimate
 $$\bigg| \int b^-_J(y) \int \gamma(x,y)a^+_Q(x) \ dm_+(x) dm_-(y)\bigg|\leq K(\theta).$$
 In particular we can evaluate   $\mu \in B^{s}_{1,1}(I^+)\otimes B^{-t}_{1,1}(I^-)$ on every  such function $\gamma$ and get the estimate
 $$|\mu(\gamma)|\leq \Vert\mu\Vert_{B^{s}_{1,1}(I^+)\otimes B^{-t}_{1,1}(I^-)} K(\gamma).$$

Since \((B^s_{1,1})^\star = B^{-s}_{\infty,\infty}\), we can generalize further. If $$\gamma: I^- \rightarrow B^{-s}_{\infty,\infty}(I^+)$$ is \(t\)-H\"older continuous with respect to the metric \(d_-\) on \(I^-\) and the norm \(B^{-s}_{\infty,\infty}\) on the codomain, then \(\mu\) can also be evaluated on \(\gamma\) in a similar manner.

\section{Transfer operator of the bilateral shift}\label{transfer}

We will study the statistical properties of the Gibbs state $\nu$ for a H\"older potential $\phi$. To this end {\it we fix, for the rest of this work}, a H\"older potential $\phi_+\colon I^+\rightarrow \mathbb{R}$ as in (\ref{coho}). Let $|\cdot|_+$ be the reference measure $m_+=m_+^{\phi_+}$ for the potential $\phi_+$ and unilateral shift $\sigma_+$. As we saw in Remark \ref{gibbs} we have that $(I^\theta,m_\theta^{\psi_\theta},\mathcal{C}_\theta)$, with $\theta\in \{+,-\}$, is a measure space with a good grid.

Secondly, we chose an {\bf arbitrary} H\"older potential $\psi_-\colon I^-\rightarrow \mathbb{R}$, and denote by  $|\cdot|_-$ be the reference measure $m_-=m_-^{\psi_-}$ for the potential $\psi_-$ and unilateral shift $\sigma_-$.

We {\it could}, for instance, pick $\psi_-=\phi_-$, where $\phi_-$ is as in  (\ref{coho}), however the fact that one can take an arbitrary $\psi_-$ makes our results far stronger.

It is  convenient  to choose metrics $d_\theta$ on $I^\theta$, with $\theta\in \{+,-\}$  that reflects nicely   measure-theoretical features  of $(I^\theta,m_\theta^{\psi_\theta},\mathcal{C}_\theta)$. 
We define $$d_\theta(x,y)=d_{I^\theta,m_\theta^{\psi_\theta},\mathcal{C}_\theta}(x,y).$$
Here $d_{I^\theta,m_\theta^{\psi_\theta},\mathcal{C}_s}$ is the pseudo-metric (that in the present setting {\it is } indeed a metric)  defined in Remark \ref{metric}. 

Note that $\phi_+\in B^{s_+}_{\infty,\infty}(I^+)$ for some $s_+> 0$, and $\psi_-\in B^{t_-}_{\infty,\infty}(I^-)$ for some $t_- > 0$.

One can see that if $\mu << m_+\!\!\times m_-$ then $\sigma^\star\mu << m_+\!\!\times m_-.$ Indeed for $h \in L^1(m_+\!\!\times m_-)$ we have
$$\sigma^\star (h \ m_+\!\!\times m_-) = (\mathcal{L}h) \ m_+\!\!\times m_-,$$
where $\mathcal{L}$ is the transfer operator

$$\mathcal{L}h (x,y)=\exp\left(\phi_+(\sigma_{+,\pi_{-1}(y)}^{-1}(x))-\psi_-(y)\right)h(\sigma_{+,\pi_{-1}(y)}(x),\sigma_-(y)).    $$

\subsection{Action on atoms} Due the atomic representation (\ref{atomic}), to study the action of $\mathcal{L}$ on $B^{s}_{1,1}(I^+)\otimes B^{-t}_{1,1}(I^-)$ it is enough to study the action on atoms, that have the form

\begin{equation}h(x,y)= a_Q(x) b_J(y), \end{equation}

To this end, we have

\begin{lemma} Let $a^+\colon I^+\rightarrow \mathbb{C}$
 be a $m_+$-integrable function and
 $b^-\colon I^-\rightarrow \mathbb{C}$
be a $m_-$-integrable function. Suppose that 
\begin{align*} 
supp \ a^+ \subset C_+(\hat{x}_0).
\end{align*}
Then 
$$\mathcal{L}(a^+ b^-)(x,y)=A^+(x)B^-(y),$$ 
where
\begin{align*} A^+(x)=&\exp\left(\phi_+(\sigma_{+,\hat{x}_0}^{-1}(x))\right)a^+(\sigma_{+,\hat{x}_0}^{-1}(x)), \\
B^-(y)=&\exp\left(-\psi_-(y)\right)b^-(\sigma_-(y))1_{C_-(\hat{x}_0)}(y).
\end{align*}
Additionally
\begin{itemize} 
\item[I.] If $$supp \ b^- \subset  R=C_-(\hat{y}_{-m},\dots,\hat{y}_{-1}) \in \mathcal{C}_-$$ 
we have 
\begin{align*} &supp \ B^- \subset  \sigma_{-,\hat{x}_0}^{-1}(R)=C_-(\hat{y}_{-m},\dots,\hat{y}_{-1},\hat{x}_0) \in \mathcal{C}_-
\end{align*}
\item[II.] If  
\begin{align*}
&supp \ a^+ \subset P=C_+(\hat{x}_0,\hat{x}_1,\dots,\hat{x}_k)\in \mathcal{C}_+\end{align*}
then 
$$supp \ A^+ \subset \sigma_{+}(P)=C_+(\hat{x}_1,\dots,\hat{x}_k)\in \mathcal{C}_+.
$$
\end{itemize}
\end{lemma}
A recursive argument gives us

\begin{prop}\label{prod2} Let $a^+\colon I^+\rightarrow \mathbb{C}$
be a $m_+$-integrable function and
 $b^-\colon I^-\rightarrow \mathbb{C}$
be a $m_-$-integrable function. Suppose that 
\begin{align*} &supp \ b^- \subset  R=C_-(\hat{y}_{-m},\dots,\hat{y}_{-1}) \in \mathcal{C}_-,\\
&supp \ a^+ \subset P=C_+(\hat{x}_0,\hat{x}_1,\dots,\hat{x}_k)\in \mathcal{C}_+
\end{align*}
Then for $\ell\leq k$ we have
$$\mathcal{L}^\ell(a^+ b^-)(x,y)=A^+(x)B^-(y),$$ 
where
\begin{align*} A^+(x)=&\exp\left(\sum_{j=0}^{\ell-1} \phi_+(\sigma_{+,\hat{x}_j,\dots,\hat{x}_{\ell-1}}^{-1}(x))\right)a^+(\sigma_{+,\hat{x}_0,\dots,\hat{x}_\ell}^{-1}(x)), \\
B^-(y)=&\exp\left(-\sum_{j=0}^{\ell-1} \psi_-(\sigma_-^j (y))\right)b^-(\sigma_-^\ell(y))1_{\sigma^{-1}_{-,\hat{x}_0,\dots,\hat{x}_{\ell-1}}(R)}(y).\end{align*} 
Moreover 
\begin{itemize}
\item[I.] For $\ell\leq k $ we have
\begin{align*}
&supp \ A^+ \subset \sigma_+^\ell (P)=C_+(\hat{x}_{\ell},\dots,\hat{x}_k)\in \mathcal{C}_+
\end{align*}
\item[II.] For $\ell\leq k+1$ we have 
\begin{align*}
&supp \ B^+ \subset \sigma^{-1}_{-,\hat{x}_0,\hat{x}_1,\dots,\hat{x}_{\ell-1}}(R)=C_-(\hat{y}_{-m},\hat{y}_{-(m-1)},\dots, \hat{y}_{-1},\hat{x}_{0},\dots,\hat{x}_{\ell-1})\in \mathcal{C}_-
\end{align*}
\end{itemize}
\end{prop}

\subsection{ Relation between bilateral and unilateral operators}
Given a potential  $\psi_s\colon I^s\rightarrow \mathbb{R}$ and a finite word 
$\hat{x}_0,\dots,\hat{x}_{\ell-1}$ denote 
$$\psi_{s,\hat{x}_0,\dots,\hat{x}_{\ell-1}}^\ell(x)=\sum_{j=0}^{\ell-1} \psi_s(\sigma_{s,\hat{x}_j,\dots,\hat{x}_{\ell-1}}^{-1}(x))=\sum_{j=0}^{\ell-1} \psi_s(\sigma_{s}^{j}(\sigma_{s,\hat{x}_0,\dots,\hat{x}_{\ell-1}}^{-1}(x))). $$

Given $s\in \{+,-\}$ and $\psi_s\colon I^s\rightarrow \mathbb{R}$ and a word $\omega=\hat{x}_0,\hat{x}_1,\dots,\hat{x}_{\ell-1}$ define
\begin{align*} \mathcal{L}_{\psi_s,\sigma_s,\omega} h(x) &= e^{\psi_{s,\omega}^k(x)}h(\sigma_{s,\omega}^{-1}(x)).
\end{align*} 
Note that 
\begin{align*} \mathcal{L}_{\psi_s,\sigma_s}^\ell h(x) &= \sum_{x_0,x_1,\dots,x_{\ell-1}} \mathcal{L}_{\psi_s,\sigma_s,x_0,x_1,\dots,x_{\ell-1}} h(x),
\end{align*}
Moreover
$$\mathcal{L}_{\psi_s,\sigma_s,\omega}\colon L^1(m_s,C^s_\omega)\rightarrow L^1(m_s,I^s)$$
is an isometry and its inverse $\mathcal{L}_{\psi_s,\sigma_s,\omega}^{-1}$ is
$$\mathcal{L}_{\psi_s,\sigma_s,\omega}^{-1}h(x)=\exp\left(-\sum_{j=0}^{\ell-1} \psi_s(\sigma_s^j(x))h(\sigma_s^\ell(x))\right)1_{C^s_\omega}(x).$$

\begin{cor}\label{prod} Let $a^+\colon I^+\rightarrow \mathbb{C}$
be a $m_+$-integrable function and
 $b^-\colon I^-\rightarrow \mathbb{C}$
be a $m_-$-integrable function. 
For every $\ell$ we have
\begin{align*} \mathcal{L}^\ell(a^+ b^-)(x,y) =\sum_{\omega\colon |\omega|=\ell }(\mathcal{L}_{\psi_-,\sigma_-,\omega}^{-1}b^-)(y)(\mathcal{L}_{\psi_+,\sigma_+,\omega}a^+)(x).&
\end{align*} 
\end{cor}

The following propositions are proven in Section \ref{expand}.

\begin{prop}[Operator $\mathcal{L}_{\sigma_+}$ acting on $B^s_{1,1}$] \label{lem2} Let $s < 1$ satisfying $s_+> s$. There is $C > 0$ such that the following holds. 
Let $P$ be an union of children of $C^+(\omega)$, with $\omega=x_0x_1\dots x_{\ell-1}$. Then for every $j\leq \ell$ we have 
\begin{equation}\label{eq1} \bigg\Vert\mathcal{L}_{+,\phi_+}^j \left( |P|_+^s \frac{1_P}{|P|_+}\right) \bigg\Vert_{B^s_{1,1}(I^+)}\leq C \frac{|P|^s_+}{|\sigma_+^j(P)|_+^{s}}.\end{equation}
In particular for  every $$P\in \mathcal{H}^+_J$$
and $J=C^+_{x_0x_1\dots x_{\ell-1}}$ we have
\begin{equation} \label{eq2}\Vert\mathcal{L}_{\phi_+,\sigma_+,x_0x_1\dots x_{j-1}}a^+_P\Vert_{B^{s}_{1,1}}\leq C_2 \frac{|P|_+^s}{|\sigma_+^j(P)|_+^s} ,\end{equation} 
where $a^+_P$ is the $(s,1)$-atom of $P$.
\end{prop}

\begin{prop}\label{actionplus3} Let $0< s < 1$ satisfying $s_+> s$. There is $C_3> 0$ such that 
for  every $$Q\in \mathcal{H}^+$$
 and $j\geq 0$ we have
\begin{align*}
\sum_{\omega\colon |\omega|=j} \Vert\mathcal{L}_{\phi_+,\sigma_+,\omega} a^+_Q\Vert_{B^s_{1,1}(I^+)} \leq C_3,
\end{align*}
where $a^+_Q$ is the $(s,1)$-atom of $Q$.
\end{prop} 

The following proposition is proven in Section \ref{contract}.

\begin{prop}[Operator $\mathcal{L}^{-1}_{\sigma_s,\omega}$ acting on $B^{-s}_{1,1}$] \label{lem4} Let $0< t < 1$ satisfying $t < t_-$. There are $\lambda_1 \in (0,1)$ and $C_1> 0$ such that 
for  every $$Q\in \mathcal{H}^-\setminus \{I^-\}$$
and for every word $\omega$ we have
\begin{equation}\label{eq6} \Vert\mathcal{L}_{\psi_-,\sigma_-,\omega}^{-1}b^-_Q\Vert_{B^{-t}_{1,1}}\leq C_1\frac{|\sigma_{-,\omega}^{-1}(Q)|_-^{t}}{|Q|_-^t},
\end{equation}
where $b^-_Q$ is the $(-t,1)$-atom of $Q$. Moreover for $b^-_{I^-}=1_{I^-}$ we have 
\begin{equation}\label{eq7} \Vert\mathcal{L}_{\psi_-,\sigma_-,\omega}^{-1}b^-_{I^-}\Vert_{B^{-t}_{1,1}}\leq C_1.
\end{equation}
\end{prop}

\subsection{The transfer operator is quasi-compact} The following is the main technical result

\begin{thm}\label{quasi} Let $s,t \in (0,1)$ such that $ s < s_+$ and $t<t_-$.  The operator 
$$\mathcal{L}\colon B^{s}_{1,1}(I^+)\otimes B^{-t}_{1,1}(I^-)\rightarrow B^{s}_{1,1}(I^+)\otimes B^{-t}_{1,1}(I^-)$$
is a norm-limited operator, that is  
$$\sup_k \Vert \mathcal{L}^k\Vert_{B^{s}_{1,1}(I^+)\otimes B^{-t}_{1,1}(I^-)}< \infty,$$
and its essential spectral radius is bounded by  
$$r_e(\mathcal{L})\leq \max \{ \exp(s M_+(\phi_+)), \exp(t M_-(\psi_-))   \}< 1,
$$
where
\begin{align*}
M_+(\phi_+)&= \sup_{\mu \ inv. prob. \ of \  \sigma_+} \int \phi_+ \ d\mu,\\
M_-(\psi_-)&= \sup_{\mu \ inv. prob. \ of \  \sigma_-} \int \psi_- \ d\mu.
\end{align*} 
\end{thm} 
\begin{proof} By the Corollary \ref{prod},  Proposition \ref{actionplus3} and Proposition \ref{lem4}  we have that for every $Q\in \mathcal{H}^+$ and $J\in \mathcal{H}^-$

\begin{align*} &\Big\Vert \mathcal{L}^\ell(a^+_J b^-_Q)\Big\Vert_{B^{s}_{1,1}(I^+)\otimes B^{-t}_{1,1}(I^-)} \\
&\leq \sum_{\omega\colon |\omega|=\ell }\Big\Vert\mathcal{L}_{\psi_-,\sigma_-,\omega}^{-1}b^-_Q\Big\Vert_{B^{-t}_{1,1}(I^-)} \Big\Vert\mathcal{L}_{\phi_+,\sigma_+,\omega}a^+_J\Big\Vert_{B^{s}_{1,1}(I^+)}\leq CC_3.
\end{align*} 
It follows that  $\mathcal{L}$ is well-defined and  it is  norm-limited. 

To obtain the estimative for the essential spectral radius we study the action of $\mathcal{L}$ on each element of the form $a^+_J b^-_Q.$\\

\noindent {\it Case I. $Q \in \mathcal{H}^-\setminus\{I^-\}$.} By  Corollary \ref{prod},  Proposition \ref{actionplus3} and Proposition \ref{lem4}  we have
\begin{align*} &\Big\Vert \mathcal{L}^\ell(a^+_J b^-_Q)\Big\Vert_{B^{s}_{1,1}(I^+)\otimes B^{-t}_{1,1}(I^-)} \\
&\leq \sum_{\omega\colon |\omega|=\ell }\Big\Vert\mathcal{L}_{\psi_-,\sigma_-,\omega}^{-1}b^-_Q\Big\Vert_{B^{-t}_{1,1}(I^-)} \Big\Vert\mathcal{L}_{\phi_+,\sigma_+,\omega}a^+_J\Big\Vert_{B^{s}_{1,1}(I^+)}\\
&\leq C_3 C_1 \sup_{\substack{Q\in \mathcal{H}^-\setminus \{I^-\}\\ |\omega|=\ell }} \frac{|\sigma_{-,\omega}^{-1}(Q)|_-^{s}}{|Q|_-^s} \leq C_3 C_1 \sup_{x\in I^-} \exp\left(\sum_{k=0}^{\ell-1} s\psi_-(\sigma^k_-x) \right).
\end{align*} 

\noindent {\it Case II. $Q=I^-$ and $J \in \mathcal{H}_+^r$, with $r \geq \ell$.} In this case $J\subset C^+(\tilde{\omega})$, with $|\tilde{\omega}|=\ell$.  By  Corollary \ref{prod}, Proposition \ref{lem2} and Proposition \ref{lem4}  we have
\begin{align*} &\Big\Vert \mathcal{L}^\ell(a^+_J b^-_Q)\Big\Vert_{B^{s}_{1,1}(I^+)\otimes B^{-t}_{1,1}(I^-)} \\
&\leq \Big\Vert\mathcal{L}_{\psi_-,\sigma_-,\tilde{\omega}}^{-1}b^-_{I^-}\Big\Vert_{B^{-t}_{1,1}(I^-)} \Big\Vert\mathcal{L}_{\phi_+,\sigma_+,\tilde{\omega}}a^+_J\Big\Vert_{B^{s}_{1,1}(I^+)}\\
&\leq C_2 C_1 \sup_{\substack{J\in \mathcal{H}_+^r\\ r\geq \ell}}  \frac{|J|_+^s}{|\sigma_+^\ell(J)|_+^s}  \leq C_2 C_1 \sup_{x\in I^+} \exp\left(\sum_{k=0}^{\ell-1} s\phi_+(\sigma^k_+x) \right).
\end{align*} 

\noindent {\it Case III. $Q=I^-$ and $J \in \mathcal{H}_+^r$, with $r < \ell$.} There is just a {\it finite number} of cases here. So this is a {\it finite rank} contribution.  

As a consequence by Nussbaum formula \cite{nuss}  for the essential spectral radius and familiar  bounded distortions arguments  (see Lemma \ref{dist}) it follows that 
\begin{align*} r_e(\mathcal{L})&\leq \liminf_\ell \max \bigg\{\sup_{x\in I^+} \exp\left(\frac{1}{\ell} \sum_{k=0}^{\ell-1} s\phi_+(\sigma^k_+x) \right), \sup_{x\in I^-} \exp\left(\frac{1}{\ell} \sum_{k=0}^{\ell-1} s\psi_-(\sigma^k_-x) \right)   \bigg\}\\
&\leq \max \{ \exp\left(s M_+(\phi_+)\right), \exp\left(s M_-(\psi_-)\right)   \}.
\end{align*} 
\end{proof} 

\section{Operator $\mathcal{L}_{\sigma_+,\omega}$ acting on $B^s_{1,1}$}
\label{expand}

\begin{lemma}\label{dist}  If $\phi_+$ is $\beta$-H\"older then there is $C > 0$ such that for every $\ell$ and finite word 
$\omega=x_0,\dots,x_{\ell-1}$ we have that 
$$\phi_{+,\omega}(x)=\phi_{+,x_0,\dots,x_{\ell-1}}^\ell(x)=\sum_{j=0}^{\ell-1} \phi_+(\sigma_{+,x_j,\dots,x_{\ell-1}}^{-1}(x)) $$
satisfies 
$$|\phi_{+,x_0,\dots,x_{\ell-1}}^\ell(x)-\phi_{+,x_0,\dots,x_{\ell-1}}^\ell(y)|\leq C d(x,y)^\beta.$$
\end{lemma} 

\begin{proof}[Proof of Proposition \ref{lem2}] 
 Let 
$$a^+= \frac{1_P}{|P|}$$
and $$A_+=\mathcal{L}_{+,\phi_+}^j \frac{1_P}{|P|}=e^{\phi_+^j} \frac{1_{\sigma_+^j(P)}}{|P|}$$ as in Proposition \ref{prod2}. We claim that 
$$\bigg\Vert A_+-\frac{1_{\sigma_+^j(P)}}{|\sigma_+^j(P)|}\bigg\Vert_{B^s_{1,1}(I^+)}\leq C |\sigma_+^j(P)|_+^{s^+-s}.
$$

For a set of positive measure $S$, denote
$$m_+(A_+,S)= \frac{1}{|S|_+} \int_S A_+ \ dm_+.$$

Given $Q=(L_Q,R_Q)\in \mathcal{H}^+_R$, let
$$\tilde{a}_Q=  (m_+(A_+,L_Q)- m_+(A_+,Q)) 1_{L_Q}+ (m_+(A_+,R_Q)- m_+(A_+,Q)) 1_{R_Q}.$$
We have
\begin{align*}
A_+ &=  m_+(A_+,\sigma_+^j(P)) 1_{\sigma_+^j(P)} + 
\sum_{\substack{Q\subset \sigma_+^j(P) \\ Q\in \mathcal{H}^+}} \tilde{a}_Q \\
&=\frac{1_{\sigma_+^j(P)}}{|\sigma_+^j(P)|} + \sum_{\substack{Q\subset \sigma_+^j(P) \\ Q\in \mathcal{H}^+}} \tilde{a}_Q  
\end{align*}
Due Proposition \ref{dist} there is $C$ such that 
$$\frac{1}{C} \frac{|P|}{|\sigma_+^j(P)|}\leq e^{\phi_{+,x_0,\dots,x_{j-1}}^j(x)}\leq C \frac{|P|}{|\sigma_+^j(P)|}
$$
for every $x \in \sigma_+^j(P)$ and  consequently
$$|A_+(x)-A_+(y)|\leq C \frac{1}{|\sigma_+^j(P)|} |Q|_+^{s_+}$$
for every $x,y \in Q\subset \sigma_+^j(P)$, with $Q\in \mathcal{H}^+$. An easy calculation shows that 
$$\tilde{a}_Q =  c_Q a^+_Q, $$
with 
$$c_Q= |Q|^{-s}_+|L_Q|_+(m_+(A_+,L_Q)- m_+(A_+,Q)),$$
so  $$|c_Q|\leq C \frac{|Q|_+}{|\sigma_+^j(P)|_+} |Q|_+^{s_+-s} $$
So
$$\sum_{\substack{Q\subset \sigma_+^j(P) \\ Q\in \mathcal{H}^+}} |c_Q|\leq C |\sigma_+^j(P)|_+^{s^+-s}.$$
This concludes the proof of the claim. It is easy to see that 
$$\bigg\Vert \frac{1_{\sigma_+^j(P)}}{|\sigma_+^j(P)|_+} \bigg\Vert_{B^s_{1,1}}\leq C|\sigma_+^j(P)|_+^{-s},$$
so (\ref{eq1}) follows. 
Let $P=(L_P,R_P)\in \mathcal{H}^+_J$. 
%we have
%\begin{align*} a_Q^+= |Q|^s \left(\frac{1_{L_Q}}{|L_Q|_+} - \frac{1_{R_Q}}{|R_Q|_+} \right),
%\end{align*}
By (\ref{eq1}) we have 
\begin{align*} &\bigg\Vert\mathcal{L}_{\phi_+,\sigma_+,x_0x_1\dots x_{j-1}}a^+_P -\frac{|P|_+^s}{|\sigma_+^j(P)|_+^s} |\sigma_+^j(P)|_+^s\left( \frac{1_{\sigma_+^j(L_P)}}{|\sigma_+^j(L_P)|_+}-
\frac{1_{\sigma_+^j(R_P)}}{|\sigma_+^j(R_P)|_+}\right)\Vert_{B^{s}_{1,1}} \\
&\leq C \frac{|P|_+^s}{|\sigma_+^j(P)|_+^{s}}|\sigma_+^j(P)|_+^{s^+}.
\end{align*} 
Note that perhaps the function
$$ \tilde{a}= |\sigma_+^j(P)|_+^s\left( \frac{1_{\sigma_+^j(L_P)}}{|\sigma_+^j(L_P)|_+}-
\frac{1_{\sigma_+^j(R_P)}}{|\sigma_+^j(R_P)|_+}\right)$$
is a not a $(s,1)$-atom, but its $m_+$-integral vanishes  and  $\sigma_+^j(P), \sigma_+^j(R_P), \sigma_+^j(L_P)$ are unions of children of the cylinder $C_+(\omega).$
In particular by the nice grid property there is $C> 0$, that depends only on the nice grid, such that one can write
$$\tilde{a}= \sum_{E\in \mathcal{H}_W^+} c_E a_E^+,$$
where $a_E$ is the $(s,1)$-atom of $E$ and $|c_E|\leq C$. This implies that 
\begin{align*} \bigg\Vert\mathcal{L}_{\phi_+,\sigma_+,x_0x_1\dots x_{j-1}}a^+_P\bigg\Vert_{B^{s}_{1,1}} 
\leq C \frac{|P|_+^s}{|\sigma_+^j(P)|_+^s} 
\end{align*} 
\end{proof}

\begin{proof}[Proof of Proposition \ref{actionplus3}]
Let $Q\in \mathcal{H}^+_W$, where $W=C(x_0,\dots,x_{\ell-1}).$
If $j\leq \ell$ we can apply  Proposition \ref{lem2} to conclude that 
\begin{align*}
\sum_{\omega\colon |\omega|=j} \big\Vert\mathcal{L}_{\phi_+,\sigma_+,\omega} a^+_Q\big\Vert_{B^s_{1,1}(I^+)}=\big\Vert\mathcal{L}_{\phi_+,\sigma_+,x_0.\dots,x_{j-1})} a^+_Q\big\Vert_{B^s_{1,1}(I^+)} \leq C \frac{|Q|_+^s}{|\sigma_+^j(Q)|_+^s}.
\end{align*}
On the other hand if $j > \ell$ there is $C >0$ (that depends only on the grid) such that we can write 
$$a_Q^+ =\sum_{\substack{J\subset Q\\ J\in \mathcal{C}_+^j}}  c_J \frac{|J|_+^{1-s}}{|Q|_+^{1-s}} |J|_+^s\frac{1_J}{|J|_+}$$
and $|c_J|\leq C$. By Lemma \ref{lem2} we have
\begin{align*} 
\sum_{\omega\colon |\omega|=j} \big\Vert\mathcal{L}_{\phi_+,\sigma_+,\omega} a^+_Q\big\Vert_{B^s_{1,1}(I^+)} &\leq 
C\sum_{\substack{J\subset Q\\ J\in \mathcal{C}_+^j}} \frac{|J|_+^{1-s}}{|Q|_+^{1-s}} \frac{|J|_+^{s}}{
|\sigma_+^j(J)|_+^{s}}\\
&\leq 
C\sum_{\substack{J\subset Q\\ J\in \mathcal{C}_+^j}} \frac{|J|_+}{|Q|_+^{1-s}} \leq C|Q|^{s}_+.
\end{align*} 
\end{proof}

\section{Operator $\mathcal{L}^{-1}_{\sigma_-,\omega}$ acting on $B^{-t}_{1,1}$}
\label{contract}

\begin{proof}[Proof of Proposition \ref{lem4}] Let $W\in \mathcal{C}_-$ and let $P$ be an union of children of $W$. Define $$b^-= \frac{1_P}{|P|},$$
and $$B_-=\mathcal{L}_{-,\psi_-,\omega}^{-1} \frac{1_P}{|P|}=e^{-\psi_{-,\omega}\circ \sigma_-^\ell} \frac{1_{\sigma_{-,\omega}^{-1}(P)}}{|P|}$$ as in Proposition \ref{prod2}. We claim that 
\begin{equation}\label{eq3} \bigg\Vert B_--\frac{1_{\sigma_{-,\omega}^{-1}(P)}}{|\sigma_{-,\omega}^{-1}(P)|}\bigg\Vert_{B^{-t}_{1,1}(I^+)}\leq C |\sigma_{-,\omega}^{-1}(P)|^t_-|P|_-^{t_-}.
\end{equation}
Indeed, using the same notation of the previous proposition, given $Q=(L_Q,R_Q)\in \mathcal{H}^-_R$, let
$$\tilde{b}_Q=  (m_-(B_-,L_Q)- m_-(B_-,Q)) 1_{L_Q}+ (m_-(B_-,R_Q)- m_-(B_-,Q)) 1_{R_Q}.$$
\begin{align*}
B_- &=  m_-(B_-,\sigma_{-,\omega}(P)) 1_{\sigma_{-,\omega}^{-1}(P)} + 
\sum_{\substack{Q\subset \sigma_{-,\omega}^{-1}(P) \\ Q\in \mathcal{H}^-}} \tilde{a}_Q \\
&=\frac{1_{\sigma_{-,\omega}^{-1}(P)}}{|\sigma_{-,\omega}^{-1}(P)|_-} + \sum_{\substack{Q\subset \sigma_{-,\omega}^{-1}(P)\\ Q\in \mathcal{H}^-}} \tilde{b}_Q  
\end{align*}
Due Proposition \ref{dist} there is $C$ such that 
$$\frac{1}{C} \frac{|P|_-}{|\sigma_{-,\omega}^{-1}(P)|_-}\leq e^{-\psi_{-,\omega}(\sigma_-^\ell(x))}\leq C \frac{|P|_-}{|\sigma_{-,\omega}^{-1}(P)|_-}
$$
for every $x \in \sigma_{-,\omega}^{-1}(P)$ and also
$$|B_-(x)-B_-(y)|\leq C \frac{1}{|\sigma_{-,\omega}^{-1}(P)|_-} |\sigma_-^\ell (Q)|_-^{t_-}$$
for every $x,y \in Q$, with $Q\subset \sigma_{-,\omega}^{-1}(P)$, with $Q\in \mathcal{H}^-$. We have 
$$\tilde{b}_Q =  c_Q b^-_Q, $$
with 
$$c_Q= |Q|^{t}_-|L_Q|_-(m_-(B_-,L_Q)- m_-(B_-,Q)),$$
so  $$|c_Q|\leq C \frac{|Q|_-}{|\sigma_{-,\omega}^{-1}(P)|_-} |Q|^t_-|\sigma_-^\ell (Q)|_-^{t_-}, $$
that implies
$$\sum_{\substack{Q\subset \sigma_{-,\omega}^{-1}(P) \\ Q\in \mathcal{H}^-}} |c_Q|\leq C |\sigma_{-,\omega}^{-1}(P)|^t_-|P|_-^{t_-}$$
This finishes the proof of the Claim.  Note that since $\sigma_{-,\omega}^{-1}(P)$ is a union of children of some cylinder in $\mathcal{C}_-$, Proposition \ref{dirac2}  implies that there is $C > 0$ (that depends only on the grid) such that   
$$\bigg\Vert\frac{1_{\sigma_{-,\omega}^{-1}(P)}}{|\sigma_{-,\omega}^{-1}(P)|}\bigg\Vert_{B^{-t}_{1,1}}\leq C$$
So we can  prove (\ref{eq7}) taking $P=I^-$ in  Claim (\ref{eq3}). 

To prove (\ref{eq7}), let $Q=(L_Q,R_Q)\in \mathcal{H}^-_W$. By Claim (\ref{eq3}) we have 
\begin{align*} &\bigg\Vert\mathcal{L}_{\psi_-,\sigma_-,\omega}b^-_Q -\frac{|\sigma_{-,\omega}^{-1}(Q)|_-^{t}}{|Q|_-^t} |\sigma_{-,\omega}^{-1}(Q)|_-^{-t}\left( \frac{1_{\sigma_{-,\omega}^{-1}(L_Q)}}{|\sigma_{-,\omega}^{-1}(L_Q)|_-}-
\frac{1_{\sigma_{-,\omega}^{-1}(R_Q)}}{|\sigma_{-,\omega}^{-1}(R_Q)|_-}\right)\bigg\Vert_{B^{-t}_{1,1}} \\
&\leq C\frac{|\sigma_{-,\omega}^{-1}(Q)|^t_-}{|Q|_-^{t}} |Q|_-^{t_-}
\end{align*} 
Note that the function
$$ \tilde{b}= |\sigma_{-,\omega}^{-1}(Q)|_-^{-t}\left( \frac{1_{\sigma_{-,\omega}^{-1}(L_Q)}}{|\sigma_{-,\omega}^{-1}(L_Q)|_-}-
\frac{1_{\sigma_{-,\omega}^{-1}(R_Q)}}{|\sigma_{-,\omega}^{-1}(R_Q)|_-}\right)$$
is a not in general an $(-t,1)$-atom, but its $m_-$-integral vanishes  and  $$\sigma_{-,\omega}^{-1}(Q), \sigma_{-,\omega}^{-1}(R_Q), \sigma_{-,\omega}^{-1}(L_Q)$$ are unions of children of the cylinder $\sigma_{-,\omega}^{-1}(W)$.
In particular by the nice grid property there is $C> 0$, that depends only on the nice grid, such that one can write
$$\tilde{b}= \sum_{E\in \mathcal{H}_W^-} c_E b_E^-,$$
where $a_E$ is the $(-t,1)$-atom of $E$ and $|c_E|\leq C$. This implies that 
$$\big\Vert\mathcal{L}_{\psi_-,\sigma_-,\omega}^{-1}b^-_Q\big\Vert_{B^{-t}_{1,1}}\leq C_1\frac{|\sigma_{-,\omega}^{-1}(Q)|_-^{t}}{|Q|_-^t}.$$
\end{proof} 

\section{Spectral gap and Exponential Decay of Correlations}
\label{decay}\label{ultima} 

\begin{prop}[Theorem \ref{main0}]\label{exp}  The unique element of the peripheral spectrum of the bounded operator $\mathcal{L}$ acting on $B^{s}_{1,1}(I^+)\otimes B^{-t}_{1,1} $is $1$, that is a simple isolated eigenvalue. Its eigenspace is generated by the Gibbs measure $\nu$, that is the unique physical measure with respect to $m_+\times m_-$. Furthermore  
$$\pi_1(\mu)=\langle \mu,1_{I^+\times I^-}\rangle \nu$$ for every $\mu \in B^{s}_{1,1}(I^+)\otimes B^{-t}_{1,1}(I^-)$, where $\pi_1$ is the spectral projection on the $1$-eigenspace.
\end{prop}
\begin{proof}
We separated the proof in several claims. \\

\noindent {\it Claim A.
Every element of the peripherical spectrum is an isolated eigenvalue and its generalized eigenspace coincides with its eigenspace. }\\

Since \begin{equation} \label{max} \sup \Vert\mathcal{L}^n\Vert_{B^{s}_{1,1}(I^+)\otimes B^{-t}_{1,1}(I^-)} < \infty,\end{equation} 
it follows that its spectral radius is at most $1$. On the other hand note that 
$$\langle\mathcal{L}^n(1_{I^+\times I^-}),1_{I^+\times I^-}\rangle= 1$$
for every $n$, so the spectrum radius is $1$. 
Since the essential spectrum radius of $\mathcal{L}$ is smaller than $1$ it follows that  every element of the peripheral spectrum (eigenvalues with nome one) is an isolated eigenvalue with finite dimensional generalized eigenspace. Its follows from (\ref{max}) that the generalized eigenspace  coincides with the  eigenspace. \\

\noindent {\it Claim B.  Every eigenvector of the peripheral spectrum is a  complex-valued finite and regular Borelian measure that is absolutely continuous with respect to the probability $\pi_1(1_{I^+\times 1^-}).$}

Here we follow an argument by  Blank, Keller and Liverani \cite{bkl} and Gou\"ezel and Liverani \cite{gl} (see also Demers,  Kiamari and Liverani \cite{impa} and Demers \cite{demers}) that applies to our setting as well. Let $\lambda$ be a eigenvalue of $\mathcal{L}$ with $|\lambda|=1$ and let $\pi_\lambda$ be the spectral projection on its eigenspace. We have
\begin{equation}\label{pi}  \pi_\lambda(h)= \lim_{n\rightarrow \infty} \frac{1}{n} \sum_{j< n} \frac{1}{\lambda^j}\mathcal{L}^jh.\end{equation} 
 
The test functions on the measure space with good grid $$(m_+\times m_-, I^+ \times I^-, \mathcal{C}_+\times \mathcal{C}_-)$$ are dense on $B^{s}_{1,1}(I^+)\otimes B^{-t}_{1,1}(I^-)$, so given $\mu$ in the $\lambda$-eigenspace there is a test function $h$ such that $\pi_\lambda(h)=\mu$. So given a test function $w$ we have
\begin{align*} \big|\langle\pi_\lambda(h),w\rangle\big|&\leq \limsup_{n\rightarrow \infty} \frac{1}{n} \sum_{j< n} \big|\langle\mathcal{L}^jh,w\rangle\big|\\
&\leq \limsup_{n\rightarrow \infty} \frac{1}{n} \sum_{j< n}  \int (\mathcal{L}^j|h|)|w| \ d(m_+\times m_-)\\
&\leq 
\limsup_{n\rightarrow \infty} \frac{1}{n} \sum_{j< n}  \int |h|\cdot |w|\circ \sigma^n \ d(m_+\times m_-)\\
&\leq |h|_{L^1(m_+\times m_-)} |w|_\infty,
\end{align*} 
Here $|\cdot|_\infty$ denotes the sup norm. Since test functions are dense in the space of  continuous functions $C(I^+\times I^-)$ we conclude that $\mu$ extends as a linear functional acting on $C(I^+\times I^-)$, so $\mu$ is a finite and regular Borelian complex-valued measure. It is easy to see that $\pi_1(u)$ is a (real-valued non-negative) invariant measure if the test function $u$ is non-negative  and $\tilde{\nu}=\pi_1(1_{I^+\times I^-})$ is a probability. Furthermore 
\begin{align*} \big|\langle\pi_\lambda(h),w\rangle\big|&\leq \limsup_{n\rightarrow \infty} \frac{1}{n} \sum_{j< n} \big|\langle\mathcal{L}^jh,w\rangle\big|\\
&\leq \limsup_{n\rightarrow \infty} \frac{1}{n} \sum_{j< n}  \int (\mathcal{L}^j|h|)|w| \ d(m_+\times m_-)\\
&\leq |h|_\infty  
\lim_{n\rightarrow \infty} \frac{1}{n} \sum_{j< n}  \langle\mathcal{L}^j(1_{I^+\times I^-}),|w|\rangle\\
&\leq |h|_\infty \langle\pi_1(1_{I^+\times I^-}),|w|\rangle,
\end{align*} 
for every test function $w$, and consequently for every continous $w$,  so $\mu=\pi_\lambda(h)$ is absolutely continous with respect  to  $\tilde{\nu}=\pi_1(1_{I^+\times I^-})$ and indeed 
$$\frac{d\mu}{d\tilde{\nu}} \in L^\infty(\nu).$$

\noindent {\it Claim C.  Let $\rho \in L^1(\tilde{\nu})$, with $\rho\geq 0$ and $\tilde{\nu}=\pi_1(1_{I^+\times I^-})$, such that $\rho \tilde{\nu}$ is $\sigma$-invariant. Then $\rho \tilde{\nu} \in B^{s}_{1,1}(I^+)\otimes B^{-t}_{1,1}(I^-)$ and it belongs to the image of $\pi_1$.}

Since $\tilde{\nu}$ is a finite Borelian regular measure there is a sequence $\rho_n$ of test functions that converges to $\rho$ on $L^1(\tilde{\nu})$. So
for every test function $\gamma$
\begin{align*}
\bigg| \langle\mathcal{L}^j(\rho_k\tilde{\nu}),\gamma\rangle-\int \gamma \rho d\tilde{\nu}\bigg|&=\bigg| \int \rho_k \gamma\circ \sigma^j \ d\tilde{\nu}-\int \gamma\circ \sigma^j  \rho d\tilde{\nu}\bigg| \\
&\leq |\rho_k-\rho|_{L^1(\tilde{\nu})} |\gamma|_\infty.
\end{align*} 
which implies that 
$$\bigg|\langle\pi_1(\rho_k\tilde{\nu}),\gamma\rangle- \int \gamma \rho d\tilde{\nu}\bigg|\leq |\rho_k-\rho|_{L^1(\tilde{\nu})} |\gamma|_\infty.$$
Since the image of $\pi_1$ is finite-dimensional this implies that $\rho\tilde{\nu}$ is in the image of $\pi_1$.\\

\noindent {\it Claim D. The eigenvalue $1$ is simple. Indeed the image of $\pi_1$ is generated by  $\tilde{\nu}=\pi_1(1_{I^+\times I^-})$, that  is the (unique) Gibbs state for the potential $\phi_+$ (and consequently for $\phi$). Moreover
$$\pi_1(\mu)=\langle \mu,1_{I^+\times I^-}\rangle \nu$$ for every $\mu \in B^{s}_{1,1}(I^+)\otimes B^{-t}_{1,1}(I^-)$}.

Given a  product of atoms $a^+_Qb^-_J$ and  a word $\omega$ we have
\begin{align} \nonumber \langle\mathcal{L}^k(a^+_Qb^-_J),1_{C_+(\omega)\times I^-}\rangle&=\int_{I^+\times I^-} a^+_Qb^-_J 1_{C_+(\omega)\times I^-}\circ \sigma^k dm_+\times m_-\\
\nonumber &=\int_{I^-} b^-_J \ dm_-\int_{I^+} a^+_Q 1_{C_+(\omega)}\circ \sigma^k_+ \ dm_+\\ 
\label{jum}
&\rightarrow_k \int_{I^-} b^-_J \ dm_-\int_{I^+} a^+_Q \ dm_+ \nu_+^{\phi_+}(C_+(\omega)).
\end{align} 
due the mixing properties of $\nu_+^{\phi_+}$. So
\begin{align*}\pi_+^\star(\pi_1(1_{I^+\times I^-}))(C_+(\omega))&= \langle\pi_1(1_{I^+\times I^-}),1_{C_+(\omega)\times I^-}\rangle\\&=
\langle\pi_1(a^+_{I^+}b^-_{I^-}),1_{C_+(\omega)\times I^-}\rangle\\
&=\nu_+^{\phi_+}(C_+(\omega)),
\end{align*} 
that implies that $\pi_1(1_{I^+\times I^-})$ is a $\sigma$-invariant probability whose marginal on the unstable direction is $\nu_+^{\phi_+}$, so $\pi_1(1_{I^+\times I^-})=\nu$. 
Furthermore 
$$\langle\pi_1(a^+_Qb^-_J),1_{C_+(\omega)\times I^-}\rangle=0$$
provided $(Q,J)\neq (I^+,I^-)$. In this case, the invariance of the complex-valued measure $\pi_1(a^+_Qb^-_J)$ implies that $\pi_1(a^+_Qb^-_J)=0$.

Consequently 
$\pi_1(\mu)=\mu(1_{I^+\times I^-}) \nu$ for every $\mu \in B^{s}_{1,1}(I^+)\otimes B^{-t}_{1,1}(I^-)$. Since there is an unique Gibbs measure it follows that they are multiple of $\tilde{\nu}$, and, due Claim C., we conclude that  $\nu$ is ergodic.

\noindent {\it Claim E. The peripheral spectrum is $\{1\}$.}\\
Indeed, it follows from (\ref{jum}) that if  $|\lambda|=1$ but $\lambda\neq 1$ we have that 
$$\langle\pi_\lambda(a^+_Qb^-_J),1_{C_+(\omega)\times I^-}\rangle=0$$
for every $(Q,J)$ and $\omega$. Since
$$\mathcal{L}\pi_\lambda(a^+_Qb^-_J)=\lambda \pi_\lambda(a^+_Qb^-_J), $$
that  implies
$$\langle\mathcal{L}\pi_\lambda(a^+_Qb^-_J),\gamma\rangle= \langle\pi_\lambda(a^+_Qb^-_J),\gamma\circ \sigma\rangle=\lambda \langle\pi_\lambda(a^+_Qb^-_J),\gamma\rangle$$
for every test function $\gamma$, we conclude that $\pi_\lambda(a^+_Qb^-_J)=0$, so $\pi_\lambda=0$.
\end{proof}

\begin{thm}\label{main4}  There is $C > 0$ and $\lambda < 1$ such that 
\begin{equation} \label{eq8} \big\Vert\mathcal{L}^\ell\mu - \langle \mu,1_{I^+\times I^-}\rangle\nu\big\Vert_{B^{s}_{1,1}(I^+)\otimes B^{-t}_{1,1}(I^-)}\leq C\lambda^\ell \Vert \mu \Vert_{B^{s}_{1,1}(I^+)\otimes B^{-t}_{1,1}(I^-)}\end{equation} 
for every  $\mu \in B^{s}_{1,1}(I^+)\otimes B^{-t}_{1,1}(I^-)$.  In particular for every  function $\gamma\colon I^+\times I^-\rightarrow \mathbb{C}$ that satisfies the assumptions in Theorem \ref{main0}.v
$$\bigg\Vert \langle \mathcal{L}^\ell\mu,\gamma\rangle - \langle \mu,1_{I^+\times I^-}\rangle \int \gamma \ d\nu\bigg\Vert_{B^{s}_{1,1}(I^+)\otimes B^{-t}_{1,1}(I^-)}\leq C\lambda^\ell K_{s,t}(\gamma)\Vert \mu \Vert_{B^{s}_{1,1}(I^+)\otimes B^{-t}_{1,1}(I^-)}.
$$
Here $K_{s,t}(\gamma)$ is as in Theorem \ref{main0}.v.
\end{thm} 
\begin{proof} The first inequality follows from  Theorem \ref{quasi} and Proposition \ref{exp}. The second one follows from Section \ref{evaluation}. 
\end{proof}

\begin{proof}[Proof of Theorem \ref{main1}] It is enough to show that 
$$\rho\  m_+\!\!\times \mu \in B^{s}_{1,1}(I^+)\otimes B^{-t}_{1,1}(I^-).$$
The assumption on $\rho$ implies that 
there is $C >0$ such that  for $\mu$-almost every $y$ we have  
%$$||\rho(\cdot,y)||_{L^1(m_+)}\leq C_1,$$
%and 
$$\rho(\cdot,y)= \sum_{Q\in \mathcal{H}^+}  c_{Q}(y) a^+_Q,$$
where 
$$\sum_{Q\in \mathcal{H}^+}  |c_{Q}(y)|\leq C_1.$$
Here
$$c_{Q}(y) =\frac{1}{|a^+_Q|_{L^2(m_+)}^2}  \int \rho(x,y) a^+_Q(x) \ dm_+(x).$$
So $Q\neq \mathcal{H}^+$ and $J \in\mathcal{H}^-$  we have
\begin{align*} c_{Q,J}&=\frac{1}{|a^+_Q|_{L^2(m_+)}^2|b^-_J|_{L^2(m_-)}^2} \int \int \rho(x,y) \ a^+_Q(x)b^-_J(y)  \ dm_+(x)d\mu(y)\\
&\leq  \frac{C}{ |J|_-^{-2t-1}} \int   |c_Q(y)|  |b^-_J(y)|  d\mu(y) \\
&\leq C \int_J   |c_Q(y)|  d\mu(y) \frac{|J|_-^{-t-1}  }{ |J|_-^{-2t-1}}\\
&\leq C |J|_-^{t} \int_J   |c_Q(y)|  d\mu(y)  \end{align*} 
That implies 
\begin{align*} \sum_{J\in \mathcal{H}^-} \sum_{Q\in \mathcal{H}^+}  |c_{Q,J}| &\leq  C \sum_{J\in \mathcal{H}^-} |J|_-^{t}  \int \sum_{Q\in \mathcal{H}^+}  |c_{Q}(y)| d\mu(y)\\
&\leq  C \sum_{J\in \mathcal{H}^-} |J|_-^{t} \int_J \Vert \rho(\cdot,y) \Vert_{B^{s}_{1,1}} d\mu(y)\\
&\leq C \int \Vert \rho(\cdot,y) \Vert_{B^{s}_{1,1}} d\mu(y) < \infty.
\end{align*} \end{proof}

\begin{proof}[Proof of Theorem \ref{main7}]  By Theorem \ref{main4} it is enough to show that 
$$U^\star(\rho m_+) \in B^{s}_{1,1}(I^+)\otimes B^{-t}_{1,1}(I^-).$$   We have that 
\begin{equation}\label{uuu} d_-(u(x),u(z))\leq C d_+(x,z)^\beta\end{equation} 
for some $C\geq 0$ and $\beta > 0$. 
Since $\rho \in B^s_{1,1}$ we have
$$\rho =  
\sum_{k\in \mathbb{N}} \sum_{P\in \mathcal{C}^k_+}  \sum_{Q\in \mathcal{H}_P^+} c_Q a^+_Q,$$
with 
$$\sum_{k\in \mathbb{N}} \sum_{P\in \mathcal{C}^k_+}  \sum_{Q\in \mathcal{H}_P^+} |c_Q| = ||\rho||_{B^s_{1,1}}< \infty.$$

%, due Theorem 15.1 in S. \cite{smania2022besov} if we define
%$$osc_1(\rho,Q)= \frac{1}{|Q|_+}\inf_{c\in \mathbb{C}} \int_Q |\rho(x)-c| \ dm_+(x)$$
%then the $B^s_{1,1}$-norm of $\rho$ is comparable to  
%$$osc^s_{1,1}(\rho)= \sum_k \sum_{Q\in \mathcal{C}_+^k} |Q|^{-s}_+ osc_1(\rho,Q).$$
Denote
$$m_+(\rho,Q)= \frac{1}{|Q|_+} \int_Q \rho \ dm_+.$$
Since $\rho \in L^{1/\epsilon}(m_+)$ the H\"older inequality implies 
\begin{equation}\label{qq} |m_+(\rho,Q)|\leq ||\rho||_{1/\epsilon}|Q|_+^{-\epsilon}\end{equation} 
Observe that 
$$\sum_{Q\in \mathcal{C}_+^k} \sum_{\substack{P\subset Q\\ P\in \mathcal{C}_+^{k+1}}} \left( m_+(\rho,P)-  m_+(\rho,Q)\right)  1_{P}=\sum_{Q\in \mathcal{C}_+^k} \sum_{R\in \mathcal{H}_Q^+} c_R a^+_R,   $$
so
\begin{equation} \label{estt} ||\sum_{Q\in \mathcal{C}_+^k} \sum_{\substack{P\subset Q\\ P\in \mathcal{C}_+^{k+1}}} \left( m_+(\rho,P)-  m_+(\rho,Q)\right)  1_{P}||_{B^s_{1,1}}=\sum_{Q\in \mathcal{C}_+^k} \sum_{R\in \mathcal{H}_Q^+} |c_R|.\end{equation}

Note that for a continuous function $\gamma$ we have

$$\int \gamma  \ dU^\star(\rho m_+)=\int \gamma(x,u(x)) \rho(x) dm_+(x).$$
For every $ Q\in \mathcal{C}_+$ choose $x_Q \in Q$. 
Let  
$$\mu_k =   \sum_{Q\in \mathcal{C}_+^k} \delta_{u(x_Q)}m_+(\rho,Q)1_{Q}, $$ 
where $\delta_{u(x_Q)}$ is a distribution on $I^-$ and $1_Q$ is an indicator function of $Q$ on $I^+$. The measure $\mu_k$ belongs to $B^{s}_{1,1}(I^+)\otimes B^{-t}_{1,1}(I^-)$ and $\mu_k$  converges to $U^\star(\rho m_+)$ in the weak-$\star$ topology (on $C(I^+\times I^-)^\star$). We claim they converge in the $B^{s}_{1,1}(I^+)\otimes B^{-t}_{1,1}(I^-)$ topology. 
Indeed 
\begin{align*} \mu_{k+1}-\mu_k &= \sum_{Q\in \mathcal{C}_+^k} \sum_{\substack{P\subset Q\\ P\in \mathcal{C}_+^{k+1}}}\left( \delta_{u(x_P)}m_+(\rho,P)-  \delta_{u(x_Q)}m_+(\rho,Q)\right) 1_{P}\\
&=\sum_{Q\in \mathcal{C}_+^k} \sum_{\substack{P\subset Q\\ P\in \mathcal{C}_+^{k+1}}}\delta_{u(x_Q)} \left( m_+(\rho,P)-  m_+(\rho,Q)\right)  1_{P}\\
&+ \sum_{Q\in \mathcal{C}_+^k} \sum_{\substack{P\subset Q\\ P\in \mathcal{C}_+^{k+1}}} (\delta_{u(x_P)}-\delta_{u(x_Q)})  m_+(\rho,P)1_{P}
\end{align*} 
By Proposition \ref{dirac2} we have
\begin{align*} \Vert\delta_{u(x_P)}-\delta_{u(x_Q)}\Vert_{B^{-t}_{1,1}}&\leq C d_-(u(x_P),u(x_Q))^t\\
&\leq C |u|_{\dot{C}^\beta(I^+,I^-)} d_+(x_P,x_Q)^{\beta t} \\
&\leq C|u|_{\dot{C}^\beta(I^+,I^-)} |Q|^{\beta t}_+.\end{align*} 
%Since $\rho$ is $t$-H\"older we have 
%$$|m_+(\rho,P)-  m_+(\rho,Q)|\leq C\Vert\rho\Vert_{B^t_{\infty,\infty}(I^+)} |Q|_+^t.$$
By Propositions \ref{pos}, \ref{dirac2} and \ref{prodd},  and additionaly  (\ref{qq}) and (\ref{estt})   we have
\begin{align*} &\sum_k \Vert\mu_{k+1}-\mu_k\Vert_{B^{s}_{1,1}(I^+)\otimes B^{-t}_{1,1}(I^-)}\\
&\leq C \sum_k \Vert \sum_{Q\in \mathcal{C}_+^k} \sum_{\substack{P\subset Q\\ P\in \mathcal{C}_+^{k+1}}} \left( m_+(\rho,P)-  m_+(\rho,Q)\right)  1_{P} \Vert_{B^s_{1,1}} \\
&+ C |u|_{\dot{C}^\beta(I^+,I^-)} \Vert\rho\Vert_{1/\epsilon} \sum_k \sum_{Q\in \mathcal{C}_+^k} \sum_{\substack{P\subset Q\\ P\in \mathcal{C}_+^{k+1}}} |Q|_+^{1+\beta t -s -\epsilon} \\
&\leq C(\Vert\rho\Vert_{B^s_{1,1}(I^+)} + |u|_{\dot{C}^\beta(I^+,I^-)}\Vert\rho\Vert_{1/\epsilon}),
\end{align*} 
So $\mu_k$ is a Cauchy sequence in $B^{s}_{1,1}(I^+)\otimes B^{-t}_{1,1}(I^-)$ and $U^\star(\rho m_+)$ belongs to  $B^{s}_{1,1}(I^+)\otimes B^{-t}_{1,1}(I^-)$. This concludes the proof of Theorem  \ref{main7}.
\end{proof}

\begin{proof}[Proof of Proposition \ref{holdermult} ]

It is enough to show that there is $C > 0$ such that for every $(Q',J')\in \mathcal{H}^+\times \mathcal{H}^-$ we have that 
$$\rho a_{Q'}^+b_{J'}^- = \sum_{Q\in \mathcal{H}^+}\sum_{J\in \mathcal{H}^-} c_{Q,J} a_{Q}^+b_{J}^-,$$
with 
$$\sum_{Q\in \mathcal{H}^+}\sum_{J\in \mathcal{H}^-} |c_{Q,J}|< C.$$
Indeed

\begin{align*} c_{Q,J}&=\frac{1}{|a^+_{Q}|_{L^2(m_+)}^2|b^-_{J}|_{L^2(m_-)}^2} \int \int \rho \ a^+_{Q'}b^-_{J'}a^+_{Q}b^-_{J}  \ dm_+dm_-.
\end{align*}

If either $J'\cap J=\emptyset$ or $Q'\cap Q=\emptyset$ then $c_{Q,J}=0$.

There are a few remaining cases\\
\noindent {\it Case 1.  If   $J\neq J'$ and $Q\neq Q'$. }
Then we have 
$$\int a_Q^+a_{Q'}^+  dm_+=\int b_J^-b_{J'}^-  dm_-=0,$$ and consequently if  we choose $x_0\in Q\cap Q'$ and $y_0\in J\cap J'$ then 
\begin{align*}
&\Big| \int \int \rho \ a^+_{Q'}b^-_{J'}a^+_{Q}b^-_{J}  \ dm_+dm_-\Big|\\
=&
\Big| \int \int [(\rho(x,y)-\rho(x,y_0))- (\rho(x_0,y)-\rho(x_0,y_0)] \ a^+_{Q'}b^-_{J'}a^+_{Q}b^-_{J}  \ dm_+(x)dm_-(y)\Big|\\
&\leq  \Big(\int |Q\cap Q'|^t |a^+_{Q'}| |a^+_{Q}|  dm_+ \Big)\Big(\int |J\cap J'|^t | b^-_{J'}| |b^-_{J}| dm_-  \Big) 
\end{align*} 
Here we used that 
\begin{align} &|(\rho(x,y)-\rho(x,y_0))- (\rho(x_0,y)-\rho(x_0,y_0)|\nonumber  \\
&\leq C \min \{d_+(x,x_0), d_-(y,y_0)\}^{2t} \nonumber \\
&\leq C d_+(x,x_0)^{t}d_-(y,y_0)^{t}.\label{difff}
\end{align} 
 If $J\subset J'$ we have 
\begin{align*} \frac{1}{|b^-_{J}|_{L^2(m_-)}^2} \int |J|^t_- |b^-_{J'}(y)||b^-_{J}(y)| dm_-(y)&\leq C \frac{|J'|^{-t-1}_-|J|^{-t-1}_-|J|_-^{1+t}}{|J|_-^{-2t-1}} \\
&\leq C \frac{|J|_-^{1+t}}{|J'|_-^{1+t}}|J|_-^t,
\end{align*} 
and if $J'\subset J$ then
\begin{align*} \frac{1}{|b^-_{J}|_{L^2(m_-)}^2} \int |J'|^t_- |b^-_{J'}(y)||b^-_{J}(y)| dm_-(y)&\leq C \frac{|J'|^{-t-1}_-|J|^{-t-1}_-|J'|_-^{1+t}}{|J|_-^{-2t-1}} \\
&\leq C |J|_-^{t}.
\end{align*} 

If $Q\subset Q'$
\begin{align*} &\frac{1}{|a^+_{Q}|_{L^2(m_+)}^2}  \int |Q|^t_+ |a^+_{Q'}(x)||a^+_{Q}(x)| dm_+(x) \\
& \leq  C\frac{|Q'|^{s-1}_+|Q|^{s-1}_+|Q|_+^{1+t}}{|Q|_+^{2s-1}}\\
& \leq  C\frac{|Q|^{1+t-s}_+}{|Q'|_+^{1-s}}
\end{align*} 

If $Q'\subset Q$ we have
\begin{align*} &\frac{1}{|a^+_{Q}|_{L^2(m_+)}^2}  \int |Q'|^t_+ |a^+_{Q'}(x)||a^+_{Q}(x)| dm_+(x) \\
& \leq  C\frac{|Q'|^{s-1}_+|Q|^{s-1}_+|Q'|_+^{1+t}}{|Q|_+^{2s-1}}\\
& \leq  C\frac{|Q'|^{s}_+}{|Q|_+^{s}} |Q'|^{t}_+
\end{align*} 
Combining all these cases one  can conclude that 
$$\sum_{\substack{ Q\in \mathcal{H}_+, \ Q\neq Q' \\J\in \mathcal{H}_-, J\neq J' }} |c_{Q,J}|< C_1< \infty,$$
where $C_1$ does not depend on $(Q',J')$.\\ 
\noindent {\it Case 2.  If   $J=J'$ and $Q\neq Q'$. } Choose $x_0\in Q\cap Q'$. Then
$$|\rho(x,y)-\rho(x_0,y)|\leq C d_+(x,x_0)^{t},$$
and we can use an argument similar to Case A. to obtain  
\begin{align*}
&\Big| \int \int \rho \ a^+_{Q'}b^-_{J'}a^+_{Q}b^-_{J'}  \ dm_+dm_-\Big|\\
&\leq  \Big(\int |Q\cap Q'|^t |a^+_{Q'}| |a^+_{Q}|  dm_+ \Big)  | b^-_{J'}|_{L^2(m_-)}^2 
\end{align*} 
and we can use the estimates for the above integral in Case 1.  to conclude that
$$\sum_{Q\in \mathcal{H}_+, \ Q\neq Q' } |c_{Q,J'}|< C_2 < \infty.$$
where $C_2$ does not depend on $(Q',J')$.

\noindent {\it Case 3.  If   $Q=Q'$ and $J\neq J'$. } Use the same argument as in Case 2, exchanging the roles of $J$ and $Q$, to obtain 
$$\sum_{J\in \mathcal{H}_-, \ J\neq J' } |c_{Q',J}|< C_3 < \infty.$$
where $C_3$ does not depends on $(Q',J')$.

\noindent {\it Case 4.  If   $Q=Q'$ and $J=J'$. }  Note that
\begin{align*}\Big| \int \int \rho \ a^+_{Q'}b^-_{J'}a^+_{Q'}b^-_{J'}  \ dm_+dm_-\Big|
&\leq  |\rho|_{L^\infty(m_+\times m_-)} |a^+_{Q'}|_{L^2(m_+)}^2 | b^-_{J'}|_{L^2(m_-)}^2 
\end{align*} 
This concludes the proof. 
\end{proof} 

\section*{Appendix: Proof of Proposition \ref{srbp}} 

The proof of the following lemma is left to the reader

\begin{lemma}\label{srb}   Let $\mu$ be an  Borel probability  on $I^+\times I^-$ such that its marginal measure $\pi_+^\star \mu$ is absolutely continuous with respect to $m_+^{\phi_+}$. Then the Gibbs state $\nu$ is the unique  {\bf forward SRB measure} with respect to $\mu$ and $\gamma(\Omega_{\nu,\mu})=1$.
\end{lemma}

\begin{proof}[Proof of Proposition \ref{srbp}.C] 
First we prove  $C.$  for $\psi_+=\phi_+$ and $\psi_-=\phi_-$. Observe that 
\begin{align*} 
&\nu(C_+(x_0,x_1,\dots,x_k)\times C_-(x_{-m},\dots,x_{-1}))\\
&=\nu(\sigma^{-m}(C_+(x_0,x_1,\dots,x_k)\times C_-(x_{-m},\dots,x_{-1})))  \\
&=\nu(C_+(x_{-m},\dots,x_{-1},x_0,x_1,\dots,x_k)\times I^-)\\
&=\nu_+^{\phi_+}(C_+(x_{-m},\dots,x_{-1},x_0,x_1,\dots,x_k)).
\end{align*}
It follows that for every 
$$x\in C_+(x_0,x_1,\dots,x_k)\times C_-(x_{-m},\dots,x_{-1}), $$
we have  
\begin{align*} \sigma^{-m}x&\in C_+(x_{-m},\dots,x_{-1},x_0,x_1,\dots,x_k)\times I^-,\\
\pi_+(x)&\in C_+(x_0,x_1,\dots,x_k), \\
\pi_-(x)&\in C_-(x_{-m},\dots,x_{-1}),
\end{align*}
and since $\phi, \phi_+$ and $\phi_-$ are cohomologous there is a bounded function $u$ such that 
\begin{align*} &\sum_{j=0}^{m+k} \phi_+(\sigma^j(\sigma^{-m}x))= \sum_{j=1}^{m} \phi_+(\sigma^{-j}(x)) + \sum_{j=0}^{k} \phi_+(\sigma^j(x))\\
&=u(\sigma^{-m}x)- u(x)+ \sum_{j=1}^{m} \phi_-(\sigma^{-j}(x)) + \sum_{j=0}^{k} \phi_+(\sigma^j_+(\pi_+(x))) \\
&=u(\sigma^{-m}x)- u(x)+ \sum_{j=1}^{m} \phi_-(\sigma^{j}_-(\pi_-(x))) + \sum_{j=0}^{k} \phi_+(\sigma^j_+(\pi_+(x)))
\end{align*} 
and we conclude that there is $C$, that does not depends on the cylinders, such that 
$$\bigg|\sum_{j=0}^{m+k} \phi_+(\sigma^j(\sigma^{-m}x))-\left( \sum_{r=0}^k \phi_+(\sigma_+^r(\pi_+(x)))+\sum_{\ell=1}^m \phi_-(\sigma_-^\ell(\pi_-(x))) \right)  \bigg|< C.$$
This implies
$$\frac{1}{C}\leq \frac{\nu(C_+(x_0,x_1,\dots,x_k)\times C_-(x_{-m},\dots,x_{-1}))}{\exp( \sum_{r=0}^k \phi_+(\sigma_+^r(\pi_+(x)))+\sum_{\ell=1}^m \phi_-(\sigma_-^\ell(\pi_-(x)))}\leq C$$
So 
\begin{align*}
&\frac{1}{C} m_+^{\phi_+}(C_+(x_0,x_1,\dots,x_k))m_-^{\phi_-}(C_-(x_{-m},\dots,x_{-1}))\\
&\leq \nu(C_+(x_0,x_1,\dots,x_k)\times C_-(x_{-m},\dots,x_{-1})) \\
&\leq C m_+^{\phi_+}(C_+(x_0,x_1,\dots,x_k))m_-^{\phi_-}(C_-(x_{-m},\dots,x_{-1})),
\end{align*}
and consequently $\nu$ is  equivalent to  $m_+^{\phi_+}\times m_-^{\phi_-}$. The ergodicity of $\nu$ implies that $$m_+^{\phi_+}~\times~m_-^{\phi_-}(\Omega_{\nu,m_+^{\phi_+}\times m_-^{\phi_-}}^+)=(\Omega_{\nu,m_+^{\phi_+}\times m_-^{\phi_-}}^-)=1.$$ We can also easily conclude that $\pi_-^*(\nu)=\nu_-^{\phi_-}$.
The general case for $C.$ follows from the observation that 
the  measures $\nu_+^{\psi_+}$, $m_+^{\psi_+}$, $\nu_+^{\phi_+}$ and $m_+^{\phi_+}$ are mutually equivalents, and $\nu_-^{\phi_-}= \nu_-^{\psi_-}$, and $\nu_-^{\psi_-}$, $m_-^{\psi_-}$, $\nu_-^{\phi_-}$ and $m_-^{\phi_-}$ are also mutually equivalents, and $\nu_-^{\phi_-}= \nu_-^{\psi_-}$.
\end{proof}
\begin{proof}[Proof of Proposition \ref{srbp}.B] 
To show $B.$ note that $m^{\psi_+}_+$ and $m^{\phi_+}_+$  are equivalents, consequently  $\pi_+^\star(m)=m^{\psi_+}_+ <<m^{\phi_+}_+$, so Lemma \ref{srb} implies that $\nu$ is a SRB measure for $m$ and $m(\Omega_{\nu,m})=1$. On  other hand we have that $m_-^{\psi_-}$ is mutually singular with respect to $m_-^{\phi_-}$, that implies that $m_+^{\psi_+}\times m_-^{\psi_-}$ is singular with respect to $m=m_+^{\phi_+}\times m_-^{\phi_-}$ and $\nu$. Since
$\nu_-^{\psi_-}\neq \nu_-^{\phi_-}=\pi_-^\star(\nu)$  there is a continuous function $\gamma\colon I^-\rightarrow \mathbb{R}$  such that 
$$\int \gamma \ d\nu_-^{\psi_-}\neq \int \gamma \ d\nu_-^{\phi_-}=\int \gamma\circ \pi_- \ d\nu.$$ 
Since $\nu_-^{\psi_-}$  and $m_-^{\psi_-}$ are equivalents it follows that for $(m_+^{\psi_+}\times m_-^{\psi_-})$-almost every point $x$ we have 
$$\lim_k \frac{1}{k} \sum_{i< k} \gamma(\pi_-(\sigma^{-i}x)) =\int \gamma \ d\nu^{\psi_-}_-,$$
so $\nu$ is not a backward SRB measure with respect to 
$m_+^{\psi_+}\times m_-^{\psi_-}$.

\end{proof}
\begin{proof}[Proof of Proposition \ref{srbp}.A]  We have that $m_+^{\psi_+}$ and $m_+^{\phi_+}$ are mutually singular, consequently   $m_+^{\psi_+}\times m_-^{\psi_-}$ is singular with respect to $m=m_+^{\phi_+}\times m_-^{\phi_-}$. On the other hand $\nu_+^{\psi_+}\neq \nu_+^{\phi_+}$, so, {\it mutatis mutandis}, we can use the same argument as in the proof of Proposition \ref{srbp}.B to conclude  
that  $\nu$ is not a forward SRB measure with respect to 
$m_+^{\psi_+}\times m_-^{\psi_-}$. \end{proof}

%\nocite{*}
\bibliographystyle{plain}
\bibliography{referencias.bib}
%\bibliography{main.bbl}

%\begin{thebibliography}{99}
%\addcontentsline{toc}{chapter}{References}

%\bibitem{Artigo} 
%D. Daigle and G. Freudenburg, \emph{A conterexample to Hilbert's Fourteenth Problem in dimension five}, J. Algebra 221 (1999), 528-535.

%\bibitem{Livro}
%G. Freudenburg, \emph{Algebraic Theory of Locally Nilpotent Derivations}, Encyclopaedia of Mathematical Sciences, Invariant Theory and Algebraic Transformation Groups VII, R. V. Gamkrelidze, V. L. Popov Subseries Editors.

%\end{thebibliography}

\end{document}